\documentclass[sn-mathphys-num]{sn-jnl}

\usepackage{graphicx}%
\usepackage{multirow}%
\usepackage{hyperref} 
\usepackage{amsmath,amssymb,amsfonts}%
\usepackage{amsthm}%
\usepackage{mathtools}
\usepackage{mathrsfs}%
\usepackage[title]{appendix}%
\usepackage{xcolor}%
\usepackage{textcomp}%
\usepackage{manyfoot}%
\usepackage{booktabs}%
\usepackage{algorithm}%
\usepackage{algorithmicx}%
\usepackage{algpseudocode}%
\usepackage{listings}%
\usepackage[shortlabels]{enumitem}
\usepackage{tikz-cd}
\newsavebox{\pullback}
\sbox\pullback{%
\begin{tikzpicture}%
\draw (0,0) -- (2ex,0ex);%
\draw (2ex,0ex) -- (2ex,2ex);%
\end{tikzpicture}}

\newcommand{\esc}[1]{\langle{#1}\rangle}

\newcommand{\mrm}[1]{\mathrm{#1}}
\DeclareMathOperator{\Ric}{Ric}
\DeclareMathOperator{\Div}{div}
\DeclareMathOperator{\grad}{grad}
\DeclareMathOperator{\tr}{tr}
\DeclareMathOperator{\id}{Id}
\DeclareMathOperator{\sgn}{sign}

\theoremstyle{thmstyleone}%
\newtheorem{theorem}{Theorem}
\newtheorem{proposition}[theorem]{Proposition}%
\newtheorem{lemma}[theorem]{Lemma} 
\newtheorem*{structures}{Structures on the spinor bundle}
\newtheorem{remark}[theorem]{Remark}%
\newtheorem{corollary}[theorem]{Corollary}

\theoremstyle{thmstyletwo}%

\theoremstyle{thmstylethree}%

\begin{document}

\title[Article Title]{Semi-Riemannian $\text{spin}^c$ manifolds carrying generalized Killing spinors and the classification of Riemannian $\text{spin}^c$ manifolds admitting a type I imaginary generalized Killing spinor}


\author{\fnm{Samuel} \sur{Lockman}} \email{samuel.lockman@mathematik.uni-regensburg.de}



\affil{\orgdiv{Department of Mathematics}, \orgname{University of Regensburg}, \postcode{93040}, \city{Regensburg}, \country{Germany}}



\abstract{We classify Riemannian $\text{spin}^c$ manifolds carrying a type I imaginary generalized Killing spinor, by explicitly constructing a parallel spinor on each leaf of the canonical foliation given by the Dirac current. We also provide a class of Riemannian $\text{spin}^c$ manifolds carrying a type II imaginary generalized Killing spinor, by considering spacelike hypersurfaces of Lorentzian $\text{spin}^c$ manifolds. We carry out much of the work in the setting of semi-Riemannian $\text{spin}^c$-manifolds carrying generalized Killing spinors, allowing us to draw conclusions in this setting as well. In this context, the Dirac current is not always a closed vector field. We circumvent this in even dimensions, by considering a modified Dirac current, which is closed in the cases when the original Dirac current is not. On the path to these results, we also study semi-Riemannian manifolds carrying closed and conformal vector fields.}




\maketitle
\newpage
\tableofcontents
\section{Introduction}
We say that a non-trivial (complex) spinor field $\varphi$ on a semi-Riemannian $\text{spin}^c$ manifold $(M,g)$ is a \textit{generalized Killing spinor} if there is some $g$-symmetric endomorphism field $A$ of $TM$ such that
\begin{align}
    \nabla^S_X\varphi = jA(X) \cdot \varphi \hspace{0.5cm} \forall X \in \Gamma(TM),
\end{align}
where $j \in \{1,i\}$. We call $A$ the \textit{Killing endomorphism} and if $j = 1$ $(j = i)$  we say that $\varphi$ is a \textit{real} (\textit{imaginary}) generalized Killing spinor. This is not the original definition from \cite{Rademacher91} of a generalized Killing spinor, but it is the one that is most frequent. The original definition only allowed for the Killing endomorphism $A$ to be of the form $A = \mu \id$, where $\mu$ is some real valued function on $M$. For this reason, we introduce the term "special generalized Killing spinor" to refer to this case. When the Killing endomorphism $A$ is a constant multiple of the identity, we simply call the spinor field a \textit{Killing spinor}.
\\
\\
Killing spinors where first considered in physics. They were originally introduced in general relativity and later appeared in supergravity theories. See the introduction of \cite{TwistorandKilling} for a full description of their origin. Killing spinors are also of purely geometric interest. The existence of a Killing spinor on a semi-Riemannian spin manifold implies that the manifold has constant scalar curvature, and in the Riemannian case, the manifold has to be Einstein \cite{TwistorandKilling}. Real Killing spinors on compact Riemannian spin manifolds appear as the limiting case of Friedrich's eigenvalue estimate for the Dirac operator \cite{Friedrich-erste}. Killing spinors are also related to supergeometry \cite{alekseevsky}. Generalized Killing spinors appear on hypersurfaces of semi-Riemannian spin and $\text{spin}^c$ manifolds carrying parallel spinors, which is a motivation for their study \cite{gen_cyl, Ammann2013, Diego-Romeo, Conti-salamon}.
\\ 
\\
Riemannian spin manifolds with parallel spinors are Ricci-flat and have special holonomy group \cite{Hitchin-harmonic, Friedrich-paralleler, Wang}. Baum classifies complete Riemannian spin manifolds carrying imaginary Killing spinors in \cite{Baum88} and \cite{Baum89}. Baum splits the classification into two cases, called type I and type II imaginary Killing spinors. For the type I case, Baum finds a suitable foliation of the manifold consisting of hypersurfaces carrying parallel spinors. From such a hypersurface, it is shown that a flow in the normal direction gives rise to a warped product, isometric with the original manifold. For the type II case, Baum shows that the manifold must be isometric to hyperbolic space. The classification of complete Riemannian spin manifolds with real Killing spinors goes back to the group of Friedrich, and their collaborators \cite{Friedrich_85, Friedrich_88, Friedrich_89, Friedrich_90, Grunewald_90}. The full classification of complete, simply connected Riemannian spin manifolds carrying real Killing spinors is due to Bär \cite{Bär93}. This is done by showing that the cone over the given manifold carries a parallel spinor. 
\\
\\
The above results have been extended by many in various directions. Among these, the following classifications are directly related to this article. Moroianu \cite{Moroianu-parallel} classifies complete, simply connected Riemannian $\text{spin}^c$ manifolds carrying parallel spinors. Moroianu shows that the manifold splits as a Riemannian product between a Kähler manifold (equipped with its canonical $\text{spin}^c$-structure) and a spin manifold carrying a parallel spinor \cite[Theorem 1.1]{Moroianu-parallel}. Shortly after Baum, Rademacher extends Baum's results to what we call special imaginary generalized Killing spinors \cite{Rademacher91}. Gro{\ss}e and Nakad have extended Rademacher's classification to the $\text{spin}^c$ setting \cite[Theorem 4.1]{GrosseNakad}. An extension of Rademacher's result to the setting of Riemannian spin manifolds carrying a type I imaginary generalized Killing spinor have been made by Leistner and Lischewski \cite[Theorem 4]{Leistner-Lischewski}.

\section{Preliminaries and notation}
This section is a summary of some preliminaries from spin and $\text{spin}^c$ geometry. The reader is directed to \cite{Lawson1989spin, geo-of-spacetime, Friedrich2000Dirac} for details.
\\
\\
For $n> 1$, let $ \theta : \text{Spin}_{\text{GL}}(n) \to \text{GL}_n^+$, be the connected double covering group of $\text{GL}_n^+$, where $\text{GL}_n^+$ is the connected component of the identity in $\text{GL}_n(\mathbb{R})$. If $n> 2$, this corresponds to the universal covering group. Let $M$ be an oriented smooth manifold of dimension $n$ and let $P_{\text{GL}^+}$ be the bundle of oriented frames on $M$. A \textit{topological spin structure} on $M$ is a $\text{Spin}_{\text{GL}}(n)$-principal bundle $P_{\text{Spin}_{\text{GL}}} \to M$, such that the associated $\text{GL}_n^+$-principal bundle $P_{\text{Spin}_{\text{GL}}} \times_\theta \text{GL}_n^+$ coincides with $P_{\text{GL}^+}$, meaning there is a bundle isomorphism $I : P_{\text{Spin}_{\text{GL}}} \times_\theta \text{GL}_n^+ \to P_{\text{GL}^+}$, covering the identity. We say that two topological spin structures $P_{\text{Spin}_{\text{GL}}}$ and $P_{\text{Spin}_{\text{GL}}}'$ are isomorphic if there is an isomorphism $P_{\text{Spin}_{\text{GL}}} \to P_{\text{Spin}_{\text{GL}}}'$ of $\text{Spin}_{\text{GL}}(n)$-principal bundles, compatible with the corresponding isomorphisms $I : P_{\text{Spin}_{\text{GL}}} \times_\theta \text{GL}_n^+ \to P_{\text{GL}^+}$ and $I' : P_{\text{Spin}_{\text{GL}}}' \times_\theta \text{GL}_n^+ \to P_{\text{GL}^+}$. Further, we say that $\text{Spin}_{\text{GL}}^c(n) := \text{Spin}_{\text{GL}}(n) \times_{\mathbb{Z}/2} S^1$ is the \textit{topological $\text{spin}^c$ group} and we note that there are two natural maps, $\lambda : \text{Spin}_{\text{GL}}^c(n) \to \text{GL}_n^+$ and $\rho : \text{Spin}_{\text{GL}}^c(n) \to S^1$, defined via $\lambda([A,z]) := \theta(A)$ and $\rho([A,z]) = z^2$. We define a \textit{topological $\text{spin}^c$ structure} on $M$ to be a $\text{Spin}_{\text{GL}}^c(n)$-principal bundle $P_{\text{Spin}_{\text{GL}}^c} \to M$ such that the associated $\text{GL}_n^+$-principal bundle $P_{\text{Spin}_{\text{GL}}^c} \times_{\lambda} \text{GL}_n^+$ coincides with $P_{\text{GL}^+}$. The isomorphism classes of such structures are defined in analogy with the aforementioned spin case. To any topological $\text{spin}^c$ structure, we have a naturally associated $S^1$-principal bundle $P_{S^1} := P_{\text{Spin}_{\text{GL}}^c} \times_{\rho} S^1$. A topological spin structure $P_{\text{Spin}_{\text{GL}}}$ gives rise to a double covering $P_{\text{Spin}_{\text{GL}}} \to P_{\text{GL}^+}$ of principal bundles. Similarly, a topological $\text{spin}^c$ structure $P_{\text{Spin}_{\text{GL}}^c}$ gives rise to a double covering $\Lambda : P_{\text{Spin}_{\text{GL}}^c} \to P_{\text{GL}^+} \oplus P_{S^1}$ of principal bundles. Note that a topological $\text{spin}$ structure $P_{\text{Spin}_{\text{GL}}}$ induces a topological $\text{spin}^c$ structure via $P_{\text{Spin}^c_{\text{GL}}} := P_{\text{Spin}_{\text{GL}}} \times_{\mathbb{Z}/2\mathbb{Z}} P_{S^1}$, where $P_{S^1}$ denotes the trivialized $S^1$-principal bundle on $M$. 
\\
\begin{remark}
    In terms of classifying spaces, a topological $\text{spin}^c$ structure on a smooth manifold $M$ of dimension $n$ is a lift $M \to \text{BSpin}^c_{\text{GL}}(n)$ of the classifying map $M \to \text{BGL}^+_n$ for $P_{\text{GL}^+}$ along the natural map $\text{BSpin}^c_{\text{GL}}(n) \to \text{BGL}^+_n$ induced by the map $\lambda : \text{Spin}^c_{\text{GL}}(n) \to \text{GL}^+_n$ from above. An isomorphism class of topological $\text{spin}^c$ structures then corresponds to a homotopy class of such lifts. The corresponding remark for topological spin structures is analogous.
\end{remark}
\vspace{2mm}
Next, let $g$ be a semi-Riemannian metric on $M$ of index $r$ such that $(M,g)$ is space and time oriented. We allow $r=0$, corresponding to the Riemannian case, but we exclude the case $r = n$. This yields a reduction $\iota: P_{\text{SO}^0} \xhookrightarrow{} P_{\text{GL}^+}$ of principal bundles along the canonical inclusion $\text{SO}^0(r,n-r) \xhookrightarrow{} \text{GL}^+$, where $\text{SO}^0(r,n-r)$ denotes the identity component of $\text{SO}(r,n-r)$. Given a topological $\text{spin}^c$ structure on $M$, we define the corresponding (geometric) \textit{$\text{spin}^c$ structure} as the pullback $P_{\text{Spin}^c}$ along $\Lambda$ and $\iota \oplus \id : P_{\text{SO}^0} \oplus P_{S^1}  \to P_{\text{GL}^+} \oplus P_{S^1}$. In other words, the principal bundle $P_{\text{Spin}^c}$ fits into the following pullback diagram
\[\begin{tikzcd}
	{P_{\text{Spin}^c}} & {P_{\text{Spin}_{\text{GL}}^c}} \\
	{P_{\text{SO}^0} \oplus P_{S^1}} & {P_{\text{GL}^+} \oplus P_{S^1}.}
	\arrow[drr, phantom, "\usebox\pullback" , very near start, color=black, from = 1-1, to= 2-2]
        \arrow[from=1-1, to=1-2]
	\arrow[from=1-1, to=2-1]
	\arrow["\Lambda", from=1-2, to=2-2]
	\arrow["{\iota \oplus \text{id}}", hook, from=2-1, to=2-2]
\end{tikzcd}\]
The definition of a spin structure is analogous. A space and time oriented semi-Riemannian manifold $(M,g)$, equipped with a $\text{spin}^c$ (spin) structure is said to be a \textit{time oriented semi-Riemannian $\text{spin}^c$ (spin) manifold}.
\\
\\
Let $(M,g)$ be a time oriented semi-Riemannian $\text{spin}^c$ manifold of dimension $n$ and index $r$. The bundle $P_{\text{spin}^c}$ is a $\text{spin}^c(r,n-r)$-principal bundle, where $\text{spin}^c(r,n-r) := \text{spin}(r,n-r) \times_{\mathbb{Z}/2\mathbb{Z}}S^1$ and where we write $\text{spin}(r,n-r)$ to denote the connected component of the identity of the $\text{spin}$ group sitting inside the Clifford algebra $C\ell_{r,n-r}$. Letting $\delta :  \text{spin}(r,n-r) \to \mrm{GL}(\Sigma)$ be the restriction of an irreducible representation of the complexified Clifford algebra $\mathbb{C}\ell_{r,n-r}$, we define $\Delta : \text{spin}^c(r,n-r) \to  \mrm{GL}(\Sigma)$ via $\Delta([s,z])v = z\delta(s)v$. The associated vector bundle $S := P_{\text{spin}^c} \times_{\Delta} \Sigma$ is called the \textit{complex spinor bundle} $S \to M$, which is a vector bundle of complex dimension $2^{\lfloor \frac{n}{2}\rfloor}$. When $n$ is even, $S$ splits as a direct sum $S = S^+ \oplus S^-$.
\vspace{2mm}
\begin{structures}\label{structures}
In any dimension, the spinor bundle $S$ carries the following structures:
\\
    \begin{enumerate}
    \item Clifford multiplication $\cdot : TM \otimes S \to S$, $X \otimes \varphi \mapsto X \cdot \varphi$, which satisfies the Clifford relation 
    \begin{equation*}
    X \cdot Y \cdot \varphi + Y \cdot X \cdot \varphi = -2g(X,Y)\varphi.
    \end{equation*}
    Using the metric, we can also define Clifford multiplication by one-forms. In case $n$ is even, Clifford multiplication maps elements $\varphi \in S^+$ to $X \cdot \varphi \in S^-$. 
    \vspace{2mm}
    \item\label{bundle metric} A hermitian bundle metric $\esc{\cdot,\cdot}$, which is positive definite when $r = 0$ and has signature $2^{\lfloor \frac{n}{2}\rfloor-1}$ when $r > 0$. It obeys the rule 
    \begin{align*}
        \esc{X \cdot \varphi, \psi} = (-1)^{r+1} \esc{\varphi, X \cdot \psi}.
    \end{align*}
    In case $n$ is even, the bundle metric makes the splitting $S^+ \oplus S^-$ orthogonal.
    \vspace{2mm}
    \item Given a connection $A \in \Omega^1(P_{S^1}, i\mathbb{R})$, we obtain a metric connection $\nabla^S$ on $S$ which satisfies 
    \begin{equation*}
        \nabla_X^S(Y \cdot \varphi) = (\nabla_X Y) \cdot \varphi + Y \cdot \nabla_X^S\varphi,
    \end{equation*} 
    where $\nabla$ denotes the Levi-Civita connection on $(M,g)$. In case $n$ is even, the splitting $S^+ \oplus S^-$ is $\nabla^S$-parallel. 
\end{enumerate}
\end{structures}


\vspace{2mm}
The curvature $R^S$ of $\nabla^S$ is related to the Ricci curvature, $\Ric$, of $g$ via
\begin{align}\label{curvature}
    \sum_{k=1}^{n}\varepsilon_ke_k\cdot R^S(X,e_k)\varphi = -\frac{1}{2}\Ric(X) \cdot \varphi + \frac{i}{2}(X \lrcorner \omega) \cdot \varphi, 
\end{align}
where $\{e_1, \hdots, e_n\}$ is a pseudo-orthonormal frame with $\varepsilon_k = g(e_k,e_k)$ and $i\omega \in \Omega^2(M, i\mathbb{R})$ denotes the curvature 2-form coming from the connection 1-form $A \in \Omega^1(P_{S^1}, i\mathbb{R})$ \cite[Page 65]{Friedrich2000Dirac}.
\\
\begin{remark}
    The reason why we only consider time oriented manifolds, is to be able to define the above bundle metric. It can not be defined unless this assumption is included, see Remark 6.2.14 of \cite{Bärnotes}.
\end{remark}
\section{Generalized Killing spinors}
Let $(M,g)$ denote a time oriented semi-Riemannian $\text{spin}^c$ manifold of dimension $n$ and index $r$, with all the associated structure from the above section. For clarity, we repeat the definition from the introduction. A non-trivial spinor field $\varphi \in \Gamma(S)$ is said to be a \textit{generalized Killing spinor}, if there is some $g$-symmetric endomorphism field $A$ of $TM$ such that
\begin{align}\label{Killing-def}
    \nabla^S_X\varphi = jA(X) \cdot \varphi \hspace{0.5cm} \forall X \in \Gamma(TM),
\end{align}
where $j \in \{1,i\}$. We call $A$ the \textit{Killing endomorphism} and if $j = 1$ $(j = i)$  we say that $\varphi$ is a \textit{real} (\textit{imaginary}) generalized Killing spinor. In case $A = \mu \id$ for some real valued function $\mu \in C^{\infty}(M)$, we say in addition that $\varphi$ is a \textit{special generalized Killing spinor} with \textit{Killing function} $j\mu$. Lastly, if $\mu$ is a real constant, we say that $\varphi$ is a \textit{Killing spinor} with \textit{Killing number} $j\mu$. If $\varphi$ is a generalized Killing spinor, we will simply say that it is \textit{special} to mean that it is a special generalized Killing spinor.
\\
\\
Given a spinor field $\varphi \in \Gamma(S)$, we define the \textit{Dirac current} $V_{\varphi} \in \Gamma(TM)$ as the vector field that is metrically equivalent to the $1$-form $\omega_{\varphi} : X \mapsto i^{r+1}\esc{\varphi, X \cdot \varphi}$. 
When $M$ has even dimension, we write $\varphi = \varphi^+ + \varphi^-$ and use the notation $\overline{\varphi}: = \varphi^+ - \varphi^-$ to define the one-form
\begin{align}\label{modified-dirac-current}
    \overline{\omega}_{\varphi} : X \mapsto \esc{\varphi, X \cdot \overline{\varphi}}.
\end{align}
Using the symmetry property $\esc{X \cdot \psi_1, \psi_2} = (-1)^{r+1}\esc{\psi_1, X \cdot \psi_2}$ mentioned in \ref{bundle metric} of \hyperref[structures]{structures on the spinor bundle}, we see that when the index $r$ is even, the above one-form is real valued, and when $r$ is odd, this one-form is imaginary valued. When $r$ is even, we define the \textit{modified Dirac current} $\overline{V}_{\varphi}$ as the metrically equivalent vector field associated to $\overline{\omega}_{\varphi}$. When $r$ is odd, we define the \textit{modified Dirac current}, $\overline{V}_{\varphi}$ as the metrically equivalent vector field to the $1$-form $i\overline{\omega}_{\varphi}$.
\\
\\
We say that a vector field is closed if its associated $1$-form is closed and we say that a vector field $V$ is conformal if it satisfies $L_Vg = \psi g$, for some function $\psi \in C^{\infty}(M)$. Recall that a vector field $V$ is closed if and only if $g(\nabla_XV, Y) = g(\nabla_YV, X)$, for all $X, Y \in TM$. Note that if $V$ is closed, then we have $(L_Vg)(X,Y) = 2g(\nabla_XV, Y)$, for all $X,Y \in TM$. 
\\
\begin{lemma}\label{closed-conformal-lemma}
    Let $\varphi \in \Gamma(S)$ be a generalized Killing spinor and suppose that either
    \begin{enumerate}
        \item the index $r$ is odd, and $j = 1$, or \label{case-odd-1}
        \item the index $r$ is even, and $j = i$, or\label{case-even-i}
        \item the index $r$ is odd, the dimension $n$ is even, and $j=i$, or \label{case-modified-2}
        \item the index $r$ is even, the dimension $n$ is even, and $j=1$.\label{case-modified}
    \end{enumerate}
    \vspace{2mm}
    Then in Case \ref{case-odd-1} and Case \ref{case-even-i}, $V_{\varphi}$ is a closed vector field with 
    \begin{align}\label{Good-Dirac-derivative}
        \nabla_XV_{\varphi} = c\esc{\varphi, \varphi}A(X),
    \end{align}
    where $c = -2$ if $r \mod{4} \in \{0,3\}$ and $c = 2$ otherwise. In Case \ref{case-modified-2} and Case \ref{case-modified}, $\overline{V}_{\varphi}$ is a closed vector field with
    \begin{align}\label{Good-modified-Dirac-derivative}
        \nabla_X\overline{V}_{\varphi} = 2\esc{\varphi, \overline{\varphi}}A(X).
    \end{align}
    If $\varphi$ is furthermore a special generalized Killing spinor with $A = \mu \id$, then in Case \ref{case-odd-1} and Case \ref{case-even-i}, $V_{\varphi}$ is conformal with $L_{V_{\varphi}}g = 2c\mu\esc{\varphi, \varphi}g$. In Case \ref{case-modified-2} and Case \ref{case-modified}, $\overline{V}_{\varphi}$ is conformal with $L_{\overline{V}_{\varphi}}g = 4\mu \esc{\varphi, \overline{\varphi}}g$
\end{lemma}
\begin{proof}
    Given a generalized Killing spinor $\varphi \in \Gamma(S)$, we compute for vector fields $X$ and $Y$ that
    \begin{align}\label{first_comp}
        g(\nabla_{X}V_{\varphi}, Y) &= X(g(V_{\varphi}, Y)) - g(V_{\varphi}, \nabla_{X}Y) \nonumber \\
        &= i^{r+1} (X(\esc{\varphi, Y \cdot \varphi}) - \esc{\varphi, \nabla_XY \cdot \varphi}) \nonumber \\
        &=i^{r+1}(\esc{\nabla_{X}^S\varphi, Y \cdot \varphi} + \esc{\varphi, \nabla_X^S (Y \cdot \varphi)} - \esc{\varphi, \nabla_XY \cdot \varphi}) \nonumber \\
        &=i^{r+1}(\esc{\nabla_{X}^S\varphi, Y \cdot \varphi} + \esc{\varphi, Y \cdot \nabla_X^S \varphi}) \nonumber \\
        &=i^{r+1}(\esc{jA(X)\cdot \varphi, Y \cdot \varphi} + \esc{\varphi, Y \cdot jA(X) \cdot \varphi}).
    \end{align}
    First, suppose we are in Case \ref{case-odd-1}, i.e. the index $r$ is odd, and the generalized Killing spinor is real. Then by (\ref{first_comp}) we get that
    \begin{align*}
        g(\nabla^S_XV_{\varphi}, Y) &=i^{r+1}(\esc{A(X) \cdot \varphi, Y \cdot \varphi} + \esc{\varphi, Y \cdot A(X) \cdot \varphi}) \\
         &= i^{r+1} (\esc{\varphi, A(X) \cdot Y \cdot \varphi} + \esc{\varphi, Y \cdot A(X) \cdot \varphi})\\
         &= -2i^{r+1}\esc{\varphi, \varphi}g(A(X), Y) \\
         &= c\esc{\varphi, \varphi}g(A(X), Y)
    \end{align*}
    which shows (\ref{Good-Dirac-derivative}). Note that since $A$ is $g$-symmetric, we have also showed that $V_{\varphi}$ is closed. Using this, the conformal statement follows straightforwardly. Case \ref{case-even-i} is obtained by a similar calculation. Indeed, we use (\ref{first_comp}) to get that
    \begin{align*}
        g(\nabla_{X}V_{\varphi}, Y) &= i^{r+1}(\esc{iA(X)\cdot \varphi, Y \cdot \varphi} + \esc{\varphi, Y \cdot iA(X) \cdot \varphi}) \\
        &=i^{r+2}(\esc{A(X)\cdot \varphi, Y \cdot \varphi} - \esc{\varphi, Y \cdot A(X) \cdot \varphi}) \\
        &=-i^{r+2}(\esc{\varphi, A(X) \cdot Y \cdot \varphi} + \esc{\varphi, Y \cdot A(X) \cdot \varphi}) \\
        &=2i^{r+2}\esc{\varphi, \varphi}g(A(X), Y) \\
        &=c\esc{\varphi, \varphi}g(A(X), Y),
    \end{align*}
    which completes Case \ref{case-even-i}. Suppose now that we are in Case \ref{case-modified-2}. We note that $\nabla_X^S\overline{\varphi} = -A(X) \cdot \overline{\varphi}$ and use this to compute that
    \begin{align*}
        g(\nabla_{X}\overline{V}_{\varphi}, Y) &= X(g(\overline{V}_{\varphi}, Y)) - g(\overline{V}_{\varphi}, \nabla_{X}Y) \\
        &=i(X(\esc{\varphi, Y \cdot \overline{\varphi}}) - \esc{\varphi, \nabla_XY \cdot \overline{\varphi}}) \\
        &=i(\esc{\nabla^S_X\varphi, Y \cdot \overline{\varphi}} + \esc{\varphi, \nabla^S_X(Y \cdot \overline{\varphi})} - \esc{\varphi, \nabla_XY \cdot \overline{\varphi}}) \\
        &= i(\esc{iA(X) \cdot \varphi, Y \cdot \overline{\varphi}} + \esc{\varphi, Y \cdot \nabla^S_X\overline{\varphi}}) \\
        &= -\esc{A(X) \cdot \varphi, Y \cdot \overline{\varphi}} - i\esc{\varphi, Y \cdot iA(X) \cdot \overline{\varphi}} \\
        &= -\esc{\varphi, A(X) \cdot Y \cdot \overline{\varphi}} - \esc{\varphi, Y \cdot A(X) \cdot \overline{\varphi}}) \\
        &= 2\esc{\varphi, \overline{\varphi}}g(A(X), Y),
    \end{align*}
    which completes Case \ref{case-modified-2}. Similarly, for Case \ref{case-modified}, we compute that
    \begin{align*}
        g(\nabla_{X}\overline{V}_{\varphi}, Y) &= X(g(\overline{V}_{\varphi}, Y)) - g(\overline{V}_{\varphi}, \nabla_{X}Y) \\
        &=X(\esc{\varphi, Y \cdot \overline{\varphi}}) - \esc{\varphi, \nabla_XY \cdot \overline{\varphi}} \\
        &=\esc{\nabla^S_X\varphi, Y \cdot \overline{\varphi}} + \esc{\varphi, \nabla^S_X(Y \cdot \overline{\varphi})} - \esc{\varphi, \nabla_XY \cdot \overline{\varphi}} \\
        &= \esc{A(X) \cdot \varphi, Y \cdot \overline{\varphi}} + \esc{\varphi, Y \cdot \nabla^S_X\overline{\varphi}} \\
        &= \esc{A(X) \cdot \varphi, Y \cdot \overline{\varphi}} - \esc{\varphi, Y \cdot A(X) \cdot \overline{\varphi}} \\
        &= -(\esc{\varphi, A(X) \cdot Y \cdot \overline{\varphi}} + \esc{\varphi, Y \cdot A(X) \cdot \overline{\varphi}}) \\
        &= 2 \esc{\varphi, \overline{\varphi}}g(A(X),Y),
    \end{align*}
    which completes the proof.
\end{proof} 
Our goal is to utilize these closed and sometimes conformal vector fields to draw geometric conclusions of the underlying manifold.
\section{The geometry of semi-Riemannian manifolds carrying closed and conformal vector fields}
Closed and conformal vector fields on semi-Riemannian manifolds have been studied by K\"uhnel and Rademacher in \cite{Kuhnel1994} and $\cite{kuhnel1997}$. In \cite{Kuhnel1994}, they show the following. Note that they are considering gradient vector fields, which are locally the same as closed vector fields. \\
\begin{lemma}[K\"uhnel \& Rademacher, 1994, \cite{Kuhnel1994}]\label{Kuhnel}
    Let $(M,g)$ be a semi-Riemannian manifold admitting a non-trivial conformal gradient field $V = \grad \psi$. Then the following holds.
    \begin{enumerate}
        \item\label{one} In a neighborhood $U$ of any point $p \in M$ with $g(V,V)_p \neq 0$, we have that $(U,g\vert_U)$ is isometric to a warped product of the form $(I \times F, \epsilon dr^2 + (\psi')^2g\vert_F)$, where $\epsilon = \text{sign}(g(V,V))$ and the submanifolds $\{t\} \times F$ corresponds to the level sets of $\psi$. Furthermore, the integral curves of $\frac{V}{|V|}$ are geodesics and the function $\psi : I \to \mathbb{R}$ satisfies $\psi'' = \epsilon \frac{\Delta \psi}{n}$, where $\Delta$ denotes the Laplacian. 
        \\
        \item\label{two} The zeroes of $V$ are isolated. For each such zero $p \in M$, there exists a normal neighborhood $U \subset T_pM$, for which there are coordinates on $U \setminus C_p$, with $C_p$ denoting the light cone, where the metric takes the form $g_{(r,x)} = \epsilon dr^2 + (\frac{\psi_{\epsilon}'(r)}{\psi_{\epsilon}''(0)})^2\Tilde{g}_x$, with $\epsilon = \text{sign} \hspace{1mm}g(V,V)$, and $\tilde{g}$ denotes the standard metric on $\{x \in \mathbb{R}^n \mid \esc{x,x}_{r,n-r} = \epsilon\}$. In these coordinates, we have that $\psi_{\epsilon}(r) = \psi(r,x)$. Furthermore, these coordinates extend to a conformally flat metric on all of $U$. 
    \end{enumerate}
\end{lemma}
\vspace{2mm}
Stronger results in the Riemannian setting have been proved by Rademacher in \cite{Rademacher91}. Our goal in this section is to strengthen the result \ref{one} of Lemma \ref{Kuhnel} above. We intend to do so by expanding on the ideas and techniques developed by Olea and Gutiérrez in \cite{Guti2003} and \cite{Guti2009}. First, we proceed without assuming completeness.
\subsection{Non-local splittings without completeness assumptions}
Throughout, we assume that $(M,g)$ is a connected semi-Riemannian manifold of dimension at least $2$.  
\\
\\
Let $V$ be a closed vector field with nowhere vanishing norm $|V| := \sqrt{|g(V,V)|}$ and put $E:= \frac{V}{|V|^2}$. Let $\Phi : D \to M$, with $D \subset \mathbb{R} \times M$, be the unique maximal flow of $E$. Recall that $E$ is said to be a complete vector field if $D = \mathbb{R} \times M$. For each $t \in \mathbb{R}$, we get a map $\Phi_t : M_t \to M_{-t}$, where $M_t := \{p \in M \mid (t,p) \in D\}$. The map $\Phi_t$ is a diffeomorphism with inverse $\Phi_{-t}$. In case $V$ is also conformal, we will also consider the maximal flow $\tilde{\Phi}$ of the vector field $\tilde{E} :=\frac{V}{|V|}$. We let $\tilde{D} \subset \mathbb{R} \times M$ denote the domain of $\tilde{\Phi}$.
\\
\\
One reason why we are interested in closed vector fields $V$ is that the distribution $V^{\perp} \subset TM$ is integrable and thus induces a foliation $\mathcal{F} = \{L_p\}_{p\in M}$ of codimension $1$ submanifolds $L_p$.
\\
\\
In \cite{Guti2003} and \cite{Guti2009} it is assumed in addition that $\tilde{E}$ is complete. This yields global results, which we will come back to later. To clarify matters in the non-complete case, we need to introduce the following definition. We say that a connected open subset $L_p' \subset L_p$ is a $\textit{uniform subleaf}$ with respect to the flow $\Phi$ ($\tilde{\Phi}$), if there exists some non-empty open interval $I$, containing $0$, such that $I \times L_p' \subset D$ ($I \times L_p' \subset \tilde{D}$). We call the pair $(L_p', I)$ a \textit{uniform pair} for $\Phi$ ($\tilde{\Phi}$). Given a uniform pair $(I,L_p')$ for $\Phi$ ($\tilde{\Phi}$), we write $\Phi^{L_p'}$ ($\tilde{\Phi}^{L_p'}$) to denote the restriction of $\Phi$ to $I \times L_p'$ ($\tilde{\Phi}$ to $I \times L_p$). This notation will be unambiguous since the interval $I$ will be clear from context. Note that the point $p$ might not be a point of $L_p'$. However, for any $q \in L_p$, we have that $L_q = L_p$, and therefore we will always assume that $p$ is chosen so that $p \in L_p'$.
\\
\\
We will sometimes use the following fact, which follows straightforwardly from the definition of a flow. Given a uniform pair $(I, L_p')$ for $\Phi$, the map $\Phi^{L_p'}$ is injective if and only if each integral curve for $E$ starting at a point on $L_p'$ never returns to $L_p'$.
\\
\\
Assume that $V$ is also conformal with $L_Vg = \psi g$. Then it follows from the definitions that $\psi = \frac{2 \Div V}{n}$ and $\nabla_XV = \frac{\Div V}{n}X$, for all $X \in TM$. By the last property, $|V|$ is constant along the leaves $L_p$. We collect some straightforward formulas which will be used later. We compute for any $ X,Y \perp V$ that
\begin{equation}\label{lie-deriv-for-tilde-E}
    (L_{\tilde{E}}g)(X,Y) = \frac{(L_Vg)(X,Y)}{|V|} = \frac{2\Div V}{n|V|}.
\end{equation}
Using the formula $\nabla_XV = \frac{\Div V}{n} X$ from above, we obtain that
\begin{equation}\label{tilde-E-deriv-|V|}
    \tilde{E}(|V|) = \frac{\Div V}{n}.
\end{equation}
We say that $\Phi$ ($\tilde{\Phi}$) \textit{preserves the foliation} $\mathcal{F}$, if for every uniform pair $(I,L_p')$ for $\Phi$ ($\tilde{\Phi}$), we have that $\Phi_t(L_p') \subset L_{\Phi_t(p)}$ ($\tilde{\Phi}_t(L_p') \subset L_{\tilde{\Phi}_t(p)}$).
\\
\begin{proposition}\label{foliated}
    Let $V$ be a closed vector field with nowhere vanishing norm and use the notation from above. Then $\Phi$ preserves the foliation $\mathcal{F}$. If $V$ is also conformal, then $\Tilde{\Phi}$ also preserves the foliation $\mathcal{F}$.
\end{proposition}
\begin{proof}
    Fix some uniform pair $(I, L_p')$ for $\Phi$. First, assume that $\Phi\vert_{I \times L_p'}$ is injective. Take some point $q \in L_p'$, with $q \neq p$. Since $L_p'$ is connected, there exists a path with embedded image $\ell \cong (0,1)$ in $L_p'$, containing both the point $p$ and the point $q$. Take a tubular neighborhood $\mathcal{U}$ of $\ell$ in $L_p'$. Since $V$ is nowhere tangent to $\mathcal{U}$ and $\Phi\vert_{I \times \hspace{0.3mm}\mathcal{U}}$ is injective, we have that $\Phi\vert_{I \times \hspace{0.3mm}\mathcal{U}}$ is a diffeomorphism onto an open submanifold $N$ of $M$ \cite[Theorem 9.20]{Lee2003}. Since $N$ has trivial de Rahm cohomology, we can write $V\vert_N = \grad h$, for some function $h : N \to \mathbb{R}$ and note that the intersection of each leaf $L_m \in \mathcal{F}$ with $N$ is either empty or a level set of $h$. Hence, if we show that $h(\Phi_t(p)) = h(\Phi_t(q))$, for all $t \in I$, then $\Phi_t(q) \in L_{\Phi_t(p)}$ and we could conclude the current case. By assumption, it is true for $t = 0$. Therefore, it suffices to show that $\frac{d}{dt}h(\Phi_t(p)) = \frac{d}{dt}h(\Phi_t(q))$ for all $t \in I$. We compute that
    \begin{align}\label{right-speed}
        \frac{d}{dt}h(\Phi_t(p)) &= dh(E_{\Phi_t(p)}) =g(V, \frac{V}{|V|^2})_{\Phi_t(p)} = \sgn{g(V,V)}.
    \end{align}
    With the same computation, we see that $\frac{d}{dt}h(\Phi_t(q)) = \sgn{g(V,V)}$ and we can conclude the case when $\Phi\vert_{I \times L_p'}$ is injective. Now, suppose instead that $\Phi\vert_{I \times L_p'}$ is not injective. Let $T \in \mathbb{R}$ be the largest positive real number such that $\Phi\vert_{(I \cap (-T,T)) \times L_p'}$ is injective. Hence, there exists an integral curve for $E$, starting at a point of $L_p'$, which returns to $L_p'$ either at $T$ or at $-T$, and we assume without loss of generality that it returns at $T$. There exists some $\varepsilon > 0$ such that $\Phi\vert_{(-\varepsilon, \varepsilon) \times \Phi_T(L_p')}$ is injective, and so by the first part of the proof it follows that every integral curve for $E$ starting at a point of $L_p'$ returns to $L_p$ at $T$. By induction, each integral curve starting at a point of $L_p'$ returns to $L_p$ exactly at $t \in\{kT\}_{k \in \mathbb{Z}} \cap I$. Hence $\Phi\vert_{(((k-1)T, (k+1)T) \hspace{0.5mm}\cap \hspace{0.5mm} I) \times L_p'}$ is injective for all $k$. It follows that $\Phi\vert_{((-T, T) \hspace{0.5mm}\cap\hspace{0.5mm} (I-kT)) \times \Phi_{kT}(L_p')}$ is injective, where we have used the notation $(a,b) - kT := (a- kT, b- kT)$. For all $k \in \mathbb{Z}$, we apply the first part of the proof to $\Phi\vert_{((-T, T) \hspace{0.5mm}\cap\hspace{0.5mm} (I-kT)) \times \Phi_{kT}(L_p')}$ and conclude that for all $t \in (-T,T) \cap (I-kT)$, we have that $\Phi_{t+kT}(L_p') = \Phi_t(\Phi_{kT}(L_p')) \subset L_{\Phi_t(\Phi_{kT}(p))} = L_{\Phi_{t+kT}(p)}$, and it follows that $\Phi$ preserves the foliation. Assume that $V$ is also conformal. Then $|V|$ is constant along the leaves of the foliation $\mathcal{F}$. We can, except for the argument that $\frac{d}{dt}h(\Phi_t(p)) = \frac{d}{dt}h(\Phi_t(q))$, use the same proof as above to show that $\tilde{\Phi}$ preserves the foliation $\mathcal{F}$. To show that $\frac{d}{dt}h(\Phi_t(p)) = \frac{d}{dt}h(\Phi_t(q))$, we use that $V$ is constant along the leaves to compute that $\frac{d}{dt}h(\Phi_t(p)) = (\sgn{g(V,V)})|V|_p = (\sgn{g(V,V)})|V|_q =   \frac{d}{dt}h(\Phi_t(q))$ and thus we may complete the proof.
\end{proof}
\begin{proposition}\label{local_isometry}
    Let $V$ be a closed vector field on $(M,g)$ with nowhere vanishing norm, and let $\Phi$ and $\tilde{\Phi}$ be given as above. Then, for each uniform pair $(L_p', I)$ for $\Phi$, the map
    \begin{align*}
        \Phi^{L_p'} : (I \times L_p', \frac{dt^2}{g(V,V)} + h_t) \to (M,g),
    \end{align*} 
    where $h_t = \Phi_t^*(g\vert_{L_{\Phi_t(p)}})$, is a local isometry. Furthermore, if $V$ is also conformal, then for each uniform pair $(L_p', I)$ for $\Tilde{\Phi}$, the map
    \begin{align*}
        \tilde{\Phi}^{L_p'} : (I \times L_p', \epsilon \hspace{0.26mm} dt^2 + f^2 g\vert_{L_p'}) \to (M,g),
    \end{align*}
    where $\epsilon = g(\tilde{E},\tilde{E})$ and $f(t) = \exp{\int_{0}^{t} \frac{\Div V}{n|V|}(s)ds}$, is a local isometry. 
\end{proposition}
\begin{proof}
    By construction, $\Phi^{L_p'} : I \times L_p' \to M$ is a local diffeomorphism relating $\frac{\partial}{\partial t}$ with $\frac{V}{|V|^2}$. By Proposition \ref{foliated}, $\Phi$ preserves the foliation induced by $V^{\perp}$ and therefore the pullback metric $(\Phi^{L_p'})^*(g)$ on $I \times L_p'$ can be written as $\frac{dt^2}{g(V,V)} + h_t$, where $h_t = \Phi_t^*(g\vert_{L_{\Phi_t(p)}})$. This proves the first part. Assume now that $V$ is also conformal. By an argument similar to the above, we get that the pullback metric $(\tilde{\Phi}^{L_p'})^*(g)$ is of the form $\epsilon \hspace{0.26mm} dt^2 + m_t$, where $m_t = \tilde{\Phi}_t^*(g\vert_{L_{\Phi_t(p)}})$. It remains to determine $m_t$. Fix $v,w \in T_x(L_p')$. We intend to show that $m_t(v,w)$ satisfies the differential equation $\frac{d}{dt}m_t(v,w) = (f^2)'m_t(v,w)$. Using (\ref{lie-deriv-for-tilde-E}), we compute that 
    \begin{align*}
        \frac{d}{dt}(m_t(v,w)) &= \frac{d}{dt}((\tilde{\Phi}_t^*g)(v,w)) \\
        &= (L_{\tilde{E}}g)((\tilde{\Phi}_t)_*v,(\tilde{\Phi}_t)_*w) \\
        &= \frac{2\Div V}{n|V|}(\tilde{\Phi}_t(x))g((\tilde{\Phi}_t)_*v, (\tilde{\Phi}_t)_*w), \\
        &= \frac{2\Div V}{n|V|}(\tilde{\Phi}_t(x))(m_t(v,w)),
    \end{align*}
    and hence the abovementioned differential equation is satisfied. Showing that $\frac{2\Div V}{n|V|}(\tilde{\Phi}_t(x))$ is only a function of $t$, completes the proof. As stated above, $|V|$ is constant along the leaves. Since $\tilde{\Phi}^{L_p'}$ is a local diffeomorphism, we have that around any point $p \in M$ there is a local frame of the form $\{\tilde{E}, \partial_1, \hdots, \partial_{n-1}\}$, where $\partial_k \perp V$ for all $k \in \{1, \hdots, n-1\}$. We use (\ref{tilde-E-deriv-|V|}) and that $|V|$ is constant along the leaves, to compute that
    \begin{align*}
        \partial_k\Div V = n\partial_k\tilde{E}(|V|) =  n \tilde{E}\partial_k(|V|) = 0.
    \end{align*}
    Thus the function $\Div V$ is constant along the leaves and consequently, so is $\frac{2\Div V}{n|V|}(\tilde{\Phi}_t(x))$. Hence we may unambiguously write $\frac{2\Div V}{n|V|}(\tilde{\Phi}_t(x)) = \frac{2\Div V}{n|V|}(t)$ and conclude the proof. 
\end{proof}
When $V$ is both closed and conformal, one can give conditions for when the above local isometry is an isometry onto its image. This is the content of the following proposition, which is a generalization of a result due to Gutierrez and Olea in \cite{Guti2003} and \cite{Guti2009}. 
\\
\begin{proposition}\label{injectivity}
    Let $V$ be a closed and conformal vector field on $(M,g)$ with nowhere vanishing norm, and let $\tilde{\Phi}$ be given as above. Further, let $(L_p', I)$ be a uniform pair for $\Tilde{\Phi}$ such that $\Div V_p \neq 0$, and write $N:= \tilde{\Phi}^{L_p'}(I \times L_p')$. If either
    \vspace{2mm}
    \begin{enumerate}
        \item $\Div(V)\vert_{N}$ is non-positive or non-negative, or \label{condition_1-injectivity}
        \item $\Ric(V,V)\vert_{N}$ is non-positive or non-negative, \label{condition_2-injectivity}
    \end{enumerate}
    \vspace{2mm}
    then the map $\tilde{\Phi}^{L_p'}$ is injective.
\end{proposition}
\begin{proof}
    The map $\tilde{\Phi}^{L_p'}$ is injective if and only if each integral curve of $\tilde{E}$, starting at some point of $L_p'$, with image contained in $N$, intersects $L_p'$ only at $t=0$. Fix an integral curve $\gamma$ of $\tilde{E}$ with $\gamma(0) = q \in L_p'$ and image contained in $N$. Since $|V|$ is constant along the leafs of $\mathcal{F}$, the function $|V|_{\gamma(t)}$ only depends on which leaf $\gamma(t)$ is in. Also, by (\ref{tilde-E-deriv-|V|}), the condition that $\Div V_p \neq 0$ means that $\tilde{E}(|V|)_q \neq 0$. Hence the function $|V|_{\gamma(t)}$ is not constant.
    \\
    \\
    First, suppose that $\Div(V)\vert_N \geq 0$. By (\ref{tilde-E-deriv-|V|}), we have that $\frac{d}{dt}|V|_{\gamma(t)} \geq 0$ and therefore $|V|_{\gamma(t)}$ is a non-constant monotone function. Hence, $\gamma(t)$ can never return to $L_p'$ for any $t \neq 0$. The proof is completely analogous when $\Div(V)\vert_N \leq 0$.  
    \\
    \\
    Secondly, using (\ref{tilde-E-deriv-|V|}), we get that $\nabla_XV = \tilde{E}(|V|)X$, so we may compute that
    \begin{align*}
        R(X,V)V = \nabla_X\nabla_VV - \nabla_V\nabla_XV - \nabla_{[X,V]}V = X(\tilde{E}(|V|))V - V(\tilde{E}(|V|))X,
    \end{align*}
    and therefore, $\Ric(V,V) = -(n-1)V(\tilde{E}(|V|))$. Assume that $\Ric(V,V)\vert_N$ is non-negative and suppose for a contradiction that there is some point $0 \neq t_0 \in I $ such that $\gamma(t_0) \in L_p'$. By our computation of $\Ric(V,V)$, the function $\frac{\partial^2}{\partial t^2}|V|_{\gamma(t)}$ is non-positive. By Proposition \ref{foliated}, the flow $\tilde{\Phi}$ preserves the foliation, and therefore we have that $|V|_{\gamma(t)} = |V|_{\gamma(t_0+t)}$ for all $t$ for which both sides are well defined. This means that we can extend the non-constant smooth function $|V|_{\gamma(\cdot)}$ to a non-constant smooth and periodic function $\mathbb{R} \to \mathbb{R}$ with period $t_0$, whose second derivative is non-positive. This is a contradiction since smooth non-constant periodic functions must have both positive and negative second derivatives. The same proof applies if we assume that $\Ric(V,V)\vert_N$ is non-positive.
\end{proof}
\subsection{Global splittings under completeness assumptions}
We proceed with the notation from the previous section. Given a flow $\Psi$ and an interval $(t_1,t_2) \subset \mathbb{R}$, we will sometimes write $\Psi_{(t_1, t_2)}(V)$, to denote $\Psi((t_1, t_2) \times V)$.
\vspace{2mm}
\begin{proposition}\label{surj}
    Let $\Psi$ be the flow of a nowhere vanishing complete vector field $W$ on a connected smooth manifold $P$, preserving a codimension $1$ foliation $\mathcal{Q} = \{Q_p\}_{p \in P}$. Then, for any leaf $Q_p \in \mathcal{Q}$, the map $\Psi\vert_{\mathbb{R} \times Q_p} : \mathbb{R} \times Q_p \to P$ is a normal covering map.
\end{proposition}
\begin{proof}
    Fix some $Q_p \in \mathcal{Q}$. First, we show that $\Psi\vert_{\mathbb{R} \times Q_p} : \mathbb{R} \times Q_p \to P$ is surjective. We will show that the image $\Psi(\mathbb{R} \times Q_p) =: N$ is both open and closed. Since the vector field $W$ is nowhere vanishing, we have that $\Psi\vert_{\mathbb{R} \times Q_p}$ is a local diffeomorphism and thus, $N$ is open. If $N^c$ is empty, then we are done. Suppose $N^c$ is not empty and take some $x \in N^c$. It follows that $\Psi(\mathbb{R} \times Q_x) \subset N^c$. Indeed, if there where points $t_1, t_2 \in \mathbb{R}$, $y_1 \in Q_x$, $y_2 \in Q_p$ such that $\Psi(t_1, y_1) = \Psi(t_2, y_2)$, then we can combine the integral curves $\Psi(-, y_1)$ and $\Psi(-, y_2)$ to obtain an integral curve from $Q_p$ to $Q_x$, which would mean that $x \in N$. Hence, $x$ is contained in the open subset $\Psi(\mathbb{R} \times Q_x)$, which is a subset of $N^c$. This holds for all points in $N^c$ and thus $N^c$ is open. Hence $\Psi\vert_{\mathbb{R} \times Q_p} : \mathbb{R} \times Q_p \to P$ is surjective. Now, we will show that every $x \in P$ has an evenly covered neighborhood. If $\Psi\vert_{\mathbb{R} \times Q_p}$ is injective, then we are done because $\Psi\vert_{\mathbb{R} \times Q_p}$ would be a diffeomorphism. Suppose therefore that $\Psi\vert_{\mathbb{R} \times Q_p}$ is not injective. Take some integral curve $\gamma$ of $W$ starting at a point of $Q_p$. Let $T \in \mathbb{R}$ be the smallest positive real number such that $\gamma(T) \in Q_p$. Since the flow $\Psi$ preserves the foliation $\mathcal{Q}$, we have that each integral curve of $W$ starting at a point on $Q_p$ will return to the leaf $Q_p$ at exactly the times $\{kT \mid k \in \mathbb{Z}\}$. Fix some point $x \in P$. Since $W$ is nowhere vanishing, there exists some $\varepsilon > 0$ such that $\Psi\vert_{(-\varepsilon, \varepsilon) \times Q_x} : (-\varepsilon, \varepsilon) \times Q_x \to \Psi((-\varepsilon, \varepsilon) \times Q_x) =: U$ is a diffeomorphism. Hence $U$ is open. Let $t_x \in \mathbb{R}$ be the smallest positive real number such that $\Psi_{t_x}(Q_p) = Q_x$. We note that 
     \begin{align*}
        \Psi((t_x +kT - \varepsilon, t_x+kT + \varepsilon) \times Q_p) &= \Psi_{(t_x +kT - \varepsilon, t_x+kT + \varepsilon)}(Q_p) \\
        &= \Psi_{(- \varepsilon, \varepsilon)} \circ  \Psi_{t_x}\circ \Psi_{kT}(Q_p) \\
        &= \Psi_{(- \varepsilon, \varepsilon)} \circ  \Psi_{t_x}(Q_p) \\
        &=\Psi_{(- \varepsilon, \varepsilon)}(Q_x) = U,
    \end{align*}
    and by the minimality assumptions on $T$ and $t_x$, we conclude that
    \begin{align*}
        (\Psi\vert_{\mathbb{R} \times Q_p})^{-1}(U) = \bigcup_{k \in \mathbb{Z}} ((t_x +kT - \varepsilon, t_x+kT + \varepsilon) \times Q_p).
    \end{align*}
Hence, $\Psi\vert_{\mathbb{R} \times Q_p}$ is a covering map.  We show that the group of deck transformations of the covering act transitively on the fibers. Take any $(t_0, p_0), (t_1, p_1) \in \mathbb{R} \times Q_p$ in the fiber over some $y \in P$. We claim that the diffeomorphism
\begin{align*}
    f : \mathbb{R} \times Q_p &\to \mathbb{R} \times Q_p \\
    (t,m) &\mapsto(t + (t_0-t_1), \Psi((t_1-t_0), m))
\end{align*}
is a deck transformation mapping $(t_1, p_1)$ to $(t_0, p_0)$. It is a deck transformation since
\begin{align*}
    \Psi(f(t,m)) &= \Psi(t + (t_0-t_1), \Psi((t_1-t_0),m)) \\
    &= \Psi_t \circ \Psi((t_0-t_1), \Psi((t_1-t_0),m)) \\
    &= \Psi_t \circ \Psi(0,m) = \Psi(t,m),
\end{align*}
and it maps $(t_1, p_1)$ to $(t_0, p_0)$ since
\begin{align*}
    f(t_1, p_1) &= (t_0, \Psi((t_1-t_0), p_1)) = (t_0, \Psi_{-t_0}\circ \Psi(t_1, p_1)) = (t_0, \Psi_{-t_0}(y)) = (t_0, p_0).
\end{align*}
Hence we conclude that $\Psi\vert_{\mathbb{R} \times Q_p}$ is a normal covering map.
\end{proof}
With this, we improve Proposition \ref{local_isometry} under the additional assumption of $E$ or $\tilde{E}$ being complete.
\vspace{2mm}
\begin{proposition}\label{covering}
     Let $V$ be a closed vector field on $(M,g)$ with nowhere vanishing norm, and let $E$, $\Phi$, $\tilde{E}$, and $\Tilde{\Phi}$ be given as in the previous section. If $E$ is complete, then for any leaf $L_p$ of $\mathcal{F}$, we have that
    \begin{align*}
        \Phi^{L_p} : (\mathbb{R} \times L_p, \frac{dt^2}{g(V,V)} + h_t) \to (M,g),
    \end{align*}
    where $h_t = \Phi_t^*(g\vert_{L_{\Phi_t(p)}})$, is a normal, semi-Riemannian covering. Assume in addition that $V$ is also conformal. If $\tilde{E}$ is complete, then for any leaf $L_p$ of $\mathcal{F}$, we have that
    \begin{align*}
        \tilde{\Phi}^{L_p} : (\mathbb{R} \times L_p, \epsilon \hspace{0.26mm} dt^2 + f^2 g\vert_{L_p}) \to (M,g),
    \end{align*}
    where $\epsilon = g(\tilde{E},\tilde{E})$ and $f(t) = \exp{\int_{0}^{t} \frac{\Div V}{n|V|}(s)ds}$ is a normal, semi-Riemannian covering. 
\end{proposition}
\begin{proof}
    The statement follows by Proposition \ref{surj} and Proposition \ref{local_isometry}.
\end{proof}
By applying Proposition \ref{injectivity} we obtain the following corollary.
\vspace{2mm}
\begin{corollary}\label{injectivity-cor}
    Let $V$ be a closed and conformal vector field on $(M,g)$ with nowhere vanishing norm. Suppose $\tilde{E}$ is complete, and $\Div V$ is not constantly zero. If either $\Div V$ or $\Ric(V,V)$ is non-positive or non-negative, then the covering $\tilde{\Phi}^{L_p}$ from Proposition \ref{covering} is an isometry for all leaves $L_p$.
\end{corollary}
\begin{proof}
    Let $p$ be such that $\Div V_p \neq 0$. By Proposition \ref{injectivity}, we have that $\tilde{\Phi}^{L_p}$ is an isometry. By Proposition \ref{foliated}, $\tilde{\Phi}$ preserves the foliation, and therefore the same is true for any other leaf $L_q \in \mathcal{F}$. 
\end{proof}
\section{Returning to generalized Killing spinors}
We will apply the results of the previous section to semi-Riemannian $\text{spin}^c$ manifolds carrying generalized Killing spinors. In the Riemannian case, we have a natural distinction of type I and type II imaginary generalized Killing spinors, and this allows us to further develop the classification. Given a type I imaginary generalized Killing spinor on a Riemannian $\text{spin}^c$ manifold, we will find an explicit parallel spinor on each leaf of the foliation given by the orthogonal complement of the Dirac current. One could make ad hoc conditions for when similar results hold in the semi-Riemannian setting, but we do not conform to this. Nevertheless, due to a result by Bohle \cite{Bohle}, we can find Killing and sometimes parallel spinors on the leaves when we restrict our attention to Killing spinors on time oriented semi-Riemannian spin manifolds.

\subsection{The semi-Riemannian $\text{spin}^c$ setting}\label{semi-rieman-set}
Let $(M,g)$ be a connected time oriented semi-Riemannian $\text{spin}^c$ manifold of index $r$ and dimension $n$. Let $\varphi$ be a generalized Killing spinor as in either of the cases of Lemma \ref{closed-conformal-lemma}. To simplify notation, we shall throughout the remainder of this section write $V$ to denote either $V_{\varphi}$, corresponding to Case \ref{case-odd-1} and Case \ref{case-even-i} of Lemma \ref{closed-conformal-lemma}, or $\overline{V}_{\varphi}$, corresponding to Case \ref{case-modified-2} and Case \ref{case-modified} of Lemma \ref{closed-conformal-lemma}. Hence, by Lemma \ref{closed-conformal-lemma}, $V$ is a closed vector field and if $\varphi$ is a special generalized Killing spinor, then $V$ is also conformal. Let $\mathcal{F} = \{L_p\}_{p \in M}$ be the foliation induced by $V^{\perp}$. Further, we write $\beta$ to denote either $\esc{\varphi, \varphi}$ or $\esc{\varphi, \overline{\varphi}}$, with the analog case distinction as for $V$. Since $\nabla_XV = c\beta A(X)$, we get that $\Div V = c\beta \tr(A)$, where in Case \ref{case-odd-1} and Case \ref{case-even-i}, $c \in \{\pm 2\}$ corresponds to the constant $c$ given in Lemma \ref{closed-conformal-lemma}, and in Case \ref{case-modified-2} and Case \ref{case-modified}, $c = 2$. If $|V|$ is nowhere vanishing, then we put $E := \frac{V}{|V|^2}$ and let $\Phi$ be the unique maximal flow of $E$. Similarly, let $\tilde{\Phi}$ be the unique maximal flow of $\tilde{E} := \frac{V}{|V|}$. 
\\
\\
The following facts are straightforward computations. In Case \ref{case-odd-1} and Case \ref{case-modified}, the function $Q_+ = \beta^2 + g(V,V)$ is locally constant. In Case \ref{case-even-i} and Case \ref{case-modified-2}, the function $Q_- = \beta^2 - g(V,V)$ is locally constant. Hence, in the corresponding cases, we have that $|V| = \sqrt{|Q_{\pm} -\beta^2|}$. If $|V|$ is nowhere vanishing, we shall write $\tilde{\beta}:= c\frac{\beta}{|V|} = c\frac{\beta}{\sqrt{|Q_{\pm} -\beta^2|}}$ and $\epsilon = \text{sign}(g(V,V))$, to simplify the formulas below.
\\ \\
Given a hypersurface and a unit normal $\xi$, we write $W^{\xi}$ to denote the Weingarten map $X \mapsto -(\nabla_X\xi)^{\perp}$. 
\\
\begin{proposition}\label{local isometry at spinors}
    Let $(M,g)$ be a time oriented semi-Riemannian $\text{spin}^c$ manifold of index $r$ and dimension $n$ carrying a generalized Killing spinor $\varphi$, with $|V|$ nowhere vanishing, such that any of the cases of Lemma \ref{closed-conformal-lemma} holds. For any leaf $L_p$, the Weingarten map $W^{\tilde{E}}$ for $L_p$ can be written as
    \begin{align}\label{Weingarten}
        W^{\tilde{E}}(X) = -\tilde{\beta}(A(X) - \epsilon g(A(X), \tilde{E})\tilde{E}).
    \end{align}
    For any uniform pair $(L_p', I)$ for $\Phi$, the map $\Phi^{L_p'} : (I \times L_p', \frac{dt^2}{g(V,V)} + h_t) \to (M,g)$, where $h_t = \Phi_t^*(g\vert_{L_{\Phi_t(p)}})$, is a local isometry. If $\varphi$ is special with $A = \mu \id$, then for any uniform pair $(L_p', I)$ for $\tilde{\Phi}$ we get the local isometry $\tilde{\Phi}^{L_p'} : (I \times L_p', \epsilon \hspace{0.26mm} dt^2 + f^2 g\vert_{L_p'}) \to (M,g)$, where 
    \begin{align}\label{eff}
        f(t) = \exp{\int_{0}^{t} \mu\tilde{\beta} ds}.
    \end{align}
    Furthermore, if the subleaf $L_p'$ is such that $(\mu \beta)_p \neq 0$ and the interval $I$ is such that either the function $\mu \beta$ or the function
    \vspace{2mm}
    \begin{itemize}
        \item
        $4(n-1)j^2\mu^2g(V,V) + 2\overline{j}i^{r+1}(nV(\mu)\esc{\varphi, \varphi} + \esc{\varphi, \grad(\mu) \cdot V \cdot \varphi})$  \text{in Case \ref{case-odd-1} and Case \ref{case-even-i}}
        \item $4(n-1)j^2\mu^2g(V,V) -2nV(\mu)\esc{\varphi, \overline{\varphi}} - 2\esc{\grad(\mu) \cdot V\cdot \varphi, \overline{\varphi}}$ \hspace{7.4mm} \text{in Case \ref{case-modified-2} and Case \ref{case-modified}}
    \end{itemize}
    \vspace{2mm}
    is non-negative or non-positive on $\tilde{\Phi}(I \times L_p')$, then the map $\tilde{\Phi}^{L_p'}$ is injective.
\end{proposition}
\begin{proof}
    We compute that 
    \begin{align*}
        \nabla_X\tilde{E} &= X(|g(V, V)|^{-\frac{1}{2}})V + |V|^{-1}\nabla_XV \\
        &= -\frac{1}{2}|g(V, V)|^{-\frac{3}{2}}2\epsilon g(\nabla_XV, V)V + |V|^{-1}c\beta A(X) \\
        &= -|g(V, V)|^{-\frac{3}{2}}c\beta\epsilon g(A(X), V)V + |V|^{-1}c\beta A(X) \\
        &=\frac{c\beta}{|V|}(A(X) - \epsilon g(A(X),\tilde{E})\tilde{E}) \\
        &= \tilde{\beta}(A(X) - \epsilon g(A(X),\tilde{E})\tilde{E}),
    \end{align*}
    and hence we obtain the formula for the Weingarten map. The statements about the local isometries follows by Lemma \ref{closed-conformal-lemma} and Proposition \ref{local_isometry}. The injectivity statement is a translation of the conditions of Proposition \ref{injectivity} into this setting. By the above, we get that $\Div V = cn\mu\beta$. To compute $\Ric(V,V)$, we first compute that
    \begin{align*}
        R^S(X,e_k)\varphi &= \nabla^S_X\nabla^S_{e_k}\varphi - \nabla^S_{e_k}\nabla^S_X\varphi - \nabla^S_{[X, e_k]}\varphi \\
        &=j (\nabla^S_X \mu e_k\cdot \varphi - \nabla^S_{e_k}\mu X \cdot \varphi - \mu [X, e_k]\cdot \varphi) \\
        &= j (X(\mu) e_k\cdot \varphi + \mu e_k \cdot \nabla^S_X\varphi - e_k(\mu) X  \cdot \varphi - \mu X \cdot \nabla^S_{e_k}\varphi) \\
        &= j (X(\mu) e_k\cdot \varphi + j\mu^2 e_k \cdot X \cdot \varphi - e_k(\mu) X  \cdot \varphi - j\mu^2 X \cdot e_k \cdot \varphi) \\
        &= j(X(\mu)e_k - e_k(\mu)X) \cdot \varphi + j^2\mu^2(e_k \cdot X \cdot \varphi - X \cdot e_k \cdot \varphi),
    \end{align*}
    and hence the left hand side of (\ref{curvature}) becomes
    \begin{align*}
        LHS &= \sum_{k=1}^{n}\varepsilon_ke_k \cdot (j(X(\mu)e_k - e_k(\mu)X) \cdot \varphi + j^2\mu^2(e_k \cdot X \cdot \varphi - X \cdot e_k \cdot \varphi)) \\
        &= -jnX(\mu)\varphi - j\grad(\mu) \cdot X\cdot \varphi + 2j^2\mu^2\sum_{k=1}^n\varepsilon_k e_k \cdot (e_k \cdot X \cdot \varphi + g(X,e_k)\varphi) \\
        &= -jnX(\mu)\varphi - j\grad(\mu)\cdot X\cdot \varphi -2(n-1)j^2\mu^2X \cdot \varphi,
    \end{align*}
    and thus we can write the identity (\ref{curvature}) as
\begin{align*}
    \Ric(X) \cdot \varphi = 2jnX(\mu)\varphi + 2j\grad(\mu) \cdot X\cdot \varphi +4(n-1)j^2\mu^2X \cdot \varphi + i(X \lrcorner \omega) \cdot \varphi.
\end{align*}
With this, we compute in Case \ref{case-odd-1} and Case \ref{case-even-i} that
\begin{align*}
    \Ric(V,V) &= g(\Ric(V), V) \\
    &= i^{r+1}\esc{\varphi, \Ric(V) \cdot \varphi} \\
    &= i^{r+1}\esc{\varphi, 2jnV(\mu)\varphi + 2j\grad(\mu)\cdot V\cdot \varphi +4(n-1)j^2\mu^2V \cdot \varphi + i(V \lrcorner \omega) \cdot \varphi} \\
    &= 4(n-1)j^2\mu^2g(V,V) + 2\overline{j}i^{r+1}(nV(\mu)\esc{\varphi, \varphi} + \esc{\varphi, \grad(\mu) \cdot V \cdot \varphi}),
\end{align*}
where in the last step we used that $\omega$ is a 2-form. Similarly, for Case \ref{case-modified-2} we compute that
\begin{align*}
    \Ric(V,V) &= g(\Ric(V), V) = i\esc{\varphi, \Ric(V) \cdot \overline{\varphi}} = i\esc{\Ric(V) \cdot \varphi, \overline{\varphi}} \\
    &= i\esc{2inV(\mu)\varphi + 2i\grad(\mu) \cdot V\cdot \varphi -4(n-1)\mu^2V \cdot \varphi + i(V \lrcorner \omega) \cdot \varphi, \overline{\varphi}} \\
    &=-4(n-1)\mu^2g(V,V) - 2\esc{nV(\mu)\varphi + \grad(\mu) \cdot V\cdot \varphi, \overline{\varphi}} \\
    &=-4(n-1)\mu^2g(V,V) - 2\esc{nV(\mu)\varphi, \overline{\varphi}} -2 \esc{\grad(\mu) \cdot V\cdot \varphi, \overline{\varphi}},
\end{align*}
and for case $\ref{case-modified}$, we get that
\begin{align*}
    \Ric(V,V) &= g(\Ric(V), V) = \esc{\varphi, \Ric(V) \cdot \overline{\varphi}} = -\esc{\Ric(V) \cdot \varphi, \overline{\varphi}} \\
    &= -\esc{2nV(\mu)\varphi + 2\grad(\mu) \cdot V\cdot \varphi +4(n-1)\mu^2V \cdot \varphi + i(V \lrcorner \omega) \cdot \varphi, \overline{\varphi}} \\
    &= 4(n-1)\mu^2g(V,V) -2\esc{nV(\mu)\varphi + \grad(\mu) \cdot V\cdot \varphi, \overline{\varphi}}\\
    &= 4(n-1)\mu^2g(V,V) -2(nV(\mu)\esc{\varphi, \overline{\varphi}} - \esc{\grad(\mu) \cdot V\cdot \varphi, \overline{\varphi}}),
\end{align*}
which finishes the proof.
\end{proof}
Applying part \ref{two} of Lemma \ref{Kuhnel}, we can also compute the metric around any zero of $V$, when $\varphi$ is a special generalized Killing spinor.
\\
\begin{proposition}
    Let $(M,g)$ be a time oriented semi-Riemannian $\text{spin}^c$ manifold of index $r$ and dimension $n$ carrying a special generalized Killing spinor $\varphi$ with Killing function $\mu$ such that any of the cases of Lemma \ref{closed-conformal-lemma} holds. Then the zeroes of $V$ are isolated. For each such zero $p \in M$, there exists a normal neighborhood $U \subset T_pM$, for which there are coordinates on $U \setminus C_p$, with $C_p$ denoting the light cone, where the metric takes the form $g_{(r,x)} = \epsilon dr^2 + (\frac{\psi_{\epsilon}'(r)}{\psi_{\epsilon}''(0)})^2\Tilde{g}_x$, with $\epsilon = \text{sign} \hspace{1mm}g(V,V)$, and $\tilde{g}$ denotes the standard metric on $\{x \in \mathbb{R}^n \mid \esc{x,x}_{r,n-r} = \epsilon\}$. In these coordinates, we have that $V = \grad \psi$, where $\psi(r,x) = \psi_{\epsilon}(r)$. Furthermore, these coordinates extend to a conformally flat metric on all of $U$.
\end{proposition}
\begin{proof}
    This follows by Lemma \ref{closed-conformal-lemma} and Part \ref{two} of Lemma \ref{Kuhnel}.
\end{proof}
\begin{remark}
    In case the Killing function is a non-zero constant, i.e. the spinor is a non-parallel Killing spinor, then we have globally that $V$ is a gradient vector field $V = \grad \psi$. A straightforward case by case computation of $\grad \beta$ yields that the function $\Psi$ can, up to addition by a constant, be computed as
    \begin{align}\label{formulas}
        \begin{tabular}{c|c|c|c}
             Case \ref{case-odd-1} & Case \ref{case-even-i} & Case \ref{case-modified-2}& Case \ref{case-modified} \\ \hline \rule{0pt}{13pt}
             $\frac{i^{r+1}\beta}{2\mu}$ & $-\frac{i^r \beta}{2\mu}$ & $\frac{\beta}{2\mu}$ &$-\frac{\beta}{2\mu}.$
        \end{tabular}
    \end{align}
\end{remark}
\vspace{2mm}
In case $E$ or $\tilde{E}$ is complete, we get the following improvement of Proposition \ref{local isometry at spinors}. 
\\
\begin{proposition}\label{global covering at spinors}
    Let $(M,g)$ be a connected time oriented semi-Riemannian $\text{spin}^c$ manifold of index $r$ and dimension $n$ carrying a generalized Killing spinor $\varphi$, with $|V|$ nowhere vanishing, such that any of the cases of Lemma \ref{closed-conformal-lemma} holds. For any leaf $L_p$, the Weingarten map $W^{\tilde{E}}$ is as in (\ref{Weingarten}). If $E$ is complete, then for any leaf $L_p$ we get that the map $\Phi^{L_p} : (\mathbb{R} \times L_p, \frac{dt^2}{g(V,V)} + h_t) \to (M,g)$, where $h_t = \Phi_t^*(g\vert_{L_{\Phi_t(p)}})$, is a normal, semi-Riemannian covering. If instead $\tilde{E}$ is complete and $\varphi$ is special with $A = \mu \id$, then for any leaf $L_p$, the map $\tilde{\Phi}^{L_p} : (\mathbb{R} \times L_p, \epsilon \hspace{0.26mm} dt^2 + f^2 g\vert_{L_p}) \to (M,g)$, where $f$ is as in (\ref{eff}), is a normal, semi-Riemannian covering. If in addition the function $\mu \beta$ is not constantly zero and either the function $\mu \beta$ or the function
    \vspace{2mm}
    \begin{itemize}
        \item
        $4(n-1)j^2\mu^2g(V,V) + 2\overline{j}i^{r+1}(nV(\mu)\esc{\varphi, \varphi} + \esc{\varphi, \grad(\mu) \cdot V \cdot \varphi})$  \text{in Case \ref{case-odd-1} and Case \ref{case-even-i}}
        \item $4(n-1)j^2\mu^2g(V,V) -2nV(\mu)\esc{\varphi, \overline{\varphi}} - 2\esc{\grad(\mu) \cdot V\cdot \varphi, \overline{\varphi}}$ \hspace{7.4mm} \text{in Case \ref{case-modified-2} and Case \ref{case-modified}}
    \end{itemize}
    \vspace{2mm}
    is non-negative or non-positive, then the map $\tilde{\Phi}^{L_p}$ is an isometry for any leaf $L_p$. 
\end{proposition}
\begin{proof}
    The computation of the Weingarten map is done in Proposition \ref{local isometry at spinors}. The statements about the coverings follow from Lemma \ref{closed-conformal-lemma} and Proposition \ref{covering}. The injectivity condition is a translation of the conditions of Corollary \ref{injectivity-cor} and is in complete analogue to the proof of Proposition \ref{local isometry at spinors}. 
\end{proof}
The above results contain some geometric information of manifolds carrying generalized Killing spinors, but we are unable to draw conclusions about what type of spinor equation one can obtain on the fiber. However, when restricted to spin manifolds with Killing spinors, we have the following result due to Bohle \cite[Thereom 6 $\&$ Theorem 7]{Bohle}.
\vspace{2mm}
\begin{theorem}[Bohle, 2003, \cite{Bohle}]\label{bohle}
    A time oriented semi-Riemannian warped product spin manifold of the form $(I \times F, \epsilon \hspace{0.26mm} dt^2 + f^2g_F)$ admits a Killing spinor with Killing number $\lambda \in \{\pm \frac{1}{2}, 0, \pm \frac{i}{2}\}$ if and only if the warping function satisfies $f'' = -4\epsilon\lambda^2 f$ and the fiber $F$ admits a Killing spinor to the Killing number $\pm \lambda_F$, where $\lambda_F^2 = \lambda^2f^2 + \epsilon \frac{1}{4}(f')^2$.
\end{theorem}
\vspace{2mm}
Note that by rescaling the metric, one can always assume that Killing numbers is of the above modulus. The function computed in (\ref{eff}) satisfies the equation $f'' = -4\epsilon (j\mu)^2 f$ in case the special generalized Killing spinor of Proposition \ref{local isometry at spinors} is a Killing spinor with Killing number $\lambda \in \{\pm \frac{1}{2}, 0, \pm \frac{i}{2}\}$. Using Theorem \ref{bohle}, we improve Proposition \ref{local isometry at spinors} and Proposition \ref{global covering at spinors} in the case of Killing spinors, by concluding that the fibers carries Killing spinors.
\vspace{2mm}
\begin{theorem}
    Let $(M,g)$ be a connected time oriented semi-Riemannian spin manifold carrying a Killing spinor $\varphi$ having Killing number $\lambda \in \{\pm\frac{1}{2}, 0, \pm\frac{i}{2}\}$, with $|V|$ nowhere vanishing such that any of the cases of Lemma \ref{closed-conformal-lemma} holds. Then for any uniform pair $(L_p', I)$ for $\tilde{\Phi}$, we have that the map $\tilde{\Phi}^{L_p'} : (I \times L_p', \epsilon \hspace{0.26mm} dt^2 + f^2 g\vert_{L_p'}) \to (M,g)$, where $f$ is as in (\ref{eff}), is an isometry onto its image. The fiber $(L_p', g\vert_{L_p'})$ carries a Killing spinor with Killing number $\pm\lambda_{L_p'}$, where $\lambda_{L_p'}^2 = \lambda^2f^2 + \epsilon \frac{1}{4}(f')^2$. 
    \\
    \\
    If $\tilde{E}$ is complete, then for any leaf $L_p$, the map $\tilde{\Phi}^{L_p} : (\mathbb{R} \times L_p, \epsilon \hspace{0.26mm} dt^2 + f^2 g\vert_{L_p}) \to (M,g)$, where $f$ is as in (\ref{eff}), is an isometry. The fiber $(L_p, g\vert_{L_p})$ carries a Killing spinor with Killing number $\pm\lambda_{L_p}$ satisfying $\lambda_{L_p}^2 = \lambda^2f^2 + \epsilon \frac{1}{4}(f')^2$.
\end{theorem}
\begin{proof}
    By proposition \ref{local isometry at spinors} we get that $\tilde{\Phi}^{L_p'}$ is a local isometry. However, since the function $4(n-1)\lambda^2g(V,V)$, is either positive or negative, we have by the last condition of Proposition \ref{local isometry at spinors} that $\tilde{\Phi}^{L_p'}$ is in fact an isometry. The statement regarding the Killing spinor on the fiber follows from Theorem \ref{bohle}. For the case when $\tilde{E}$ is complete, one argues identically, using Proposition \ref{global covering at spinors}.
\end{proof}
In the Lorentzian case, we have stronger results due to Bohle \cite[Section 5]{Bohle} and Leitner \cite[Theorem 5.3]{Leitner_2003}. In the Riemannian case, the global statement is essentially the same as due to Baum in \cite{Baum89}. The author is not aware of similar local statements in the Riemannian case but is sure that they are known to experts. There are also stronger results for $r =2$, due to Shafiee and Bahrampour \cite{shafiee_Bahrampour}. In case of parallel spinors we have stronger results in \cite{Friedrich-paralleler, baum-kath, Wang, Hitchin-harmonic}.

\subsection{The Riemannian classification}
In this section, we specialize to the Riemannian setting, showing how one can generalize previous classifications made by Gro{\ss}e and Nakad \cite{GrosseNakad}, and Leistner and Lischewski \cite{Leistner-Lischewski}. 

\subsubsection{Type I imaginary generalized Killing spinors}
We say that an imaginary generalized Killing spinor $\varphi$ on a Riemannian $\text{spin}^c$ manifold is of \textit{type I} if the locally constant function $Q_- = \esc{\varphi, \varphi}^2 - g(V_{\varphi}, V_{\varphi})$ equals $0$ everywhere. If $Q_-$ does not equal $0$, then we say that $\varphi$ is of \textit{type II}. By the Cauchy-Schwarz inequality, we always have that $Q_-$ is non-negative. It is a general fact that generalized Killing spinors have no zeroes. Indeed, a generalized Killing spinor is parallel with respect to the modified connection $\tilde{\nabla}^S_X \psi:= \nabla^S_{X}\psi - jA(X) \cdot \psi$ and thus if $\varphi$ had one zero, it would be zero throughout by $\tilde{\nabla}^S$-parallel transport. With this, we see in the Riemannian case that type I imaginary generalized Killing spinors are such that $|V_{\varphi}| \neq 0$. This is the first reason why we are able to classify type I imaginary generalized Killing spinors. As in Section \ref{semi-rieman-set}, let $\Phi$ be the flow of $E := \frac{V_{\varphi}}{|V_{\varphi}|^2}$, let $\tilde{\Phi}$ to be the flow of $\tilde{E} := \frac{V_{\varphi}}{|V_{\varphi}|}$, and let $ \mathcal{F} = \{L_p\}_{p \in M}$ be the foliation induced by $E^{\perp}$. Recall from Proposition \ref{foliated} that $\Phi$ preserves the foliation $\mathcal{F}$ and if $\varphi$ is special, then $\tilde{\Phi}$ also preserves the foliation $\mathcal{F}$. The following lemma shows the second reason why we are able to classify type I imaginary generalized Killing spinors. The proof of the following lemma was pointed out to me by Jonathan Glöckle and is essentially the same computation as in section 1.3 of \cite{Glöckle-thesis}.
\vspace{2mm}
\begin{lemma}\label{type I and constraint}
    Given a spinor field $\varphi$ on a Riemannian $\text{spin}^c$ manifold such that $\esc{\varphi, \varphi}^2 - g(V_{\varphi}, V_{\varphi}) = 0$ everywhere, we also have that
    \begin{align}\label{algebraic constraint}
        V_{\varphi} \cdot \varphi = i|V_{\varphi}| \varphi.
    \end{align}
    Conversely, if $\varphi$ is an imaginary generalized Killing spinor which is not parallel on any component of $M$, and satisfies equation (\ref{algebraic constraint}), then $\varphi$ is of type I.
\end{lemma}
\begin{proof}
    We compute that:
    \begin{align}\label{J-computation}
        \esc{V_{\varphi} \cdot \varphi - i|V_{\varphi}| \varphi, V_{\varphi} \cdot \varphi - i|V_{\varphi}| \varphi} &= \esc{V_{\varphi} \cdot \varphi, V_{\varphi} \cdot \varphi} -2 \hspace{1mm}\text{Re}\esc{V_{\varphi} \cdot \varphi, i|V_{\varphi}| \varphi} + \esc{i|V_{\varphi}| \varphi, i|V_{\varphi}| \varphi} \nonumber\\
        &= g(V_{\varphi}, V_{\varphi})\esc{\varphi, \varphi} -2i|V_{\varphi}|\esc{\varphi, V_{\varphi} \cdot \varphi} + |V_{\varphi}|^2 \esc{\varphi, \varphi}  \nonumber\\
        &= 2g(V_{\varphi}, V_{\varphi})\esc{\varphi, \varphi} -2|V_{\varphi}|g(V_{\varphi}, V_{\varphi}) \nonumber\\
        &= 2|V_{\varphi}|^2(\esc{\varphi, \varphi} - |V_{\varphi}|).
    \end{align}
    First, assume that $\varphi$ is a type I imaginary generalized Killing spinor. Then (\ref{algebraic constraint}) follows because both $g(\cdot, \cdot)$ and $\esc{\cdot, \cdot}$ are positive definite. Suppose now that $\varphi$ is an imaginary generalized Killing spinor which is not parallel on any component of $M$ such that $V_{\varphi} \cdot \varphi = i|V_{\varphi}| \varphi$. Then, for each point $x \in M$, we have by (\ref{J-computation}), that either $V_{\varphi}(x) = 0$, or that $\esc{\varphi, \varphi}(x) - |V_{\varphi}|(x) = 0$. Since the function $\esc{\varphi, \varphi}^2 - g(V_{\varphi}, V_{\varphi})$ is locally constant, it suffices to show that the set $N := \{x \in M \mid V_{\varphi}(x) = 0\}$ does not contain any component of $M$. Suppose for a contradiction that $N$ contains some component $M_1$ of $M$. By assumption, there is some point $x \in M_1$ and a vector $Y \in T_xM_1$ such that $A(Y) \neq 0$. Pick any local extension $\tilde{Y}$ of $Y$ to compute that
    \begin{align*}
        0 &= Y(g(V_{\varphi}, A(\tilde{Y})) = Y(i\esc{\varphi, A(\tilde{Y}) \cdot \varphi}) \\
        &= i(\esc{\nabla_Y \varphi, A(\tilde{Y}) \cdot \varphi} + \esc{\varphi, \nabla_Y A(\tilde{Y}) \cdot \varphi} + \esc{\varphi, A(\tilde{Y}) \cdot \nabla_Y \varphi}) \\
        &=i(\esc{iA(Y) \cdot \varphi, A(Y) \cdot \varphi} + \esc{\varphi, A(Y) \cdot iA(Y) \cdot \varphi}) \\
        &= 2\esc{\varphi, A(Y) \cdot A(Y) \cdot \varphi} = -2g(A(Y), A(Y)) \esc{\varphi, \varphi}_x,
    \end{align*}
    and since generalized Killing spinors have no zeroes, we conclude that $A(Y) =0$ which establishes the contradiction and completes the proof.
\end{proof}
\begin{remark}
    To see that the assumption that $\varphi$ is not allowed to be parallel on any component of $M$ in Lemma \ref{type I and constraint} is necessary, consider a complete and simply connected spin manifold $(M,g)$ that carries a parallel spinor and has irreducible holonomy. In this case, the Dirac current would have to vanish identically, because otherwise the manifold $(M,g)$ would carry a nowhere vanishing parallel vector field and the holonomy of $(M,g)$ would be reducible. Hence the Dirac current vanishes but $\esc{\varphi, \varphi}$ does not, and so $\varphi$ is of type II but satisfies equation (\ref{algebraic constraint}) trivially.
\end{remark}
\vspace{2mm}
\begin{remark}
    Baum introduced the terms type I and type II imaginary Killing spinors in \cite{Baum89}. It is shown in \cite{Baum89} and \cite{Baum88} that a type I imaginary Killing spinor on a Riemannian spin manifold satisfies (\ref{algebraic constraint}). This proof can be adapted to imaginary generalized Killing spinors as well, but is very different from the above proof of Lemma \ref{type I and constraint}. In \cite{Leistner-Lischewski}, Leistner and Lischewski uses equation (\ref{algebraic constraint}) directly.
\end{remark}
\vspace{2mm}
In what follows, we will apply the $\text{spin}^c$ Gau{\ss} formula from \cite[Section 3]{Nakad2011}. We follow the setup of that paper. Given a leaf $L_p \in \mathcal{F}$, we equip it with the induced $\text{spin}^c$ structure and we write $S_{L_p}$ to denote the induced spinor bundle on the hypersurface $L_p$.
When the dimension $n$ of $M$ is odd, we have that $S_{L_p} \cong S\vert_{L_p}$ and when $n$ is even, we have that $S_{L_p} \cong S^+\vert_{L_p}$. Hence, when $n$ is odd, we can simply write $\psi\vert_{L_p}$ to denote the restriction of a section, but when $n$ is even we write $\psi\vert_{L_p}$, to mean $\psi^+\vert_{L_p}$. In this setup, the $S^1$-bundle associated with the $\text{spin}^c$ structure on $L_p$ is simply the restriction $P_{S^1}\vert_{L_p}$, where $P_{S^1}$ is the $S^1$-bundle associated to the $\text{spin}^c$ structure on $M$. Hence, given a connection 1-form $A \in \Omega^1(P_{S^1}, i\mathbb{R})$, we get induced connections $\nabla^S$ on $S$ and $\nabla^{S_{L_p}}$ on $S_{L_p}$. If we let $\bullet : T(L_p) \otimes S_{L_p} \to S_{L_p}$ denote the Clifford multiplication on $S_{L_p}$, we have that $X \bullet \eta = \tilde{E} \cdot X \cdot \eta$, where on the right hand side, $\eta$ is either viewed as an element of $S\vert_{L_p}$ or $S^+\vert_{L_p}$, depending on if $n$ is odd or even. 
\vspace{2mm}
\begin{remark}
    It is important to point out that the above identifications of spinor bundles and Clifford multiplication on hypersurfaces is not the only one. It boils down to the choice of inclusion $C\ell_{n-1} \to C\ell_n$ of Clifford algebras. The choice which is made in \cite{Nakad2011} is the inclusion $e_j \to e_n \cdot e_j$, where $\{e_1, \hdots, e_{n}\}$ is the standard basis for $\mathbb{R}^n$ and thus also the set of standard generators for $C\ell_n$. Another choice that can be made is to choose the inclusion $e_j \to e_j$. The second choice would have worked equally well for the main results of this article. However, the below formula (\ref{spinc-gauss}) would have had to been adjusted accordingly. 
\end{remark}
\vspace{2mm}
Translating Proposition 3.3 of \cite{Nakad2011}, to our notation, we get for any spinor-field $\psi \in \Gamma(S)$ and any vector field $X \in \Gamma(T(L_p))$ that
\begin{align}\label{spinc-gauss}
    \nabla^{S_{L_p}}_X(\psi\vert_{S_{L_p}}) = (\nabla^S_X\psi)\vert_{L_p} + \frac{1}{2}W^{\tilde{E}}(X) \bullet \psi\vert_{L_p}.
\end{align}
With this, we use (\ref{Weingarten}) to compute for any type I imaginary generalized Killing spinor $\varphi$ on a Riemannian $\text{spin}^c$ manifold, that the spinor $\varphi\vert_{L_p}$, satisfies
\begin{align}\label{restriction-computation}
        \nabla^{S_{L_p}}_X\varphi\vert_{L_p} &= (\nabla^S_X\varphi)\vert_{L_p} + \frac{1}{2}W^{\tilde{E}}(X) \bullet \varphi\vert_{L_p} \nonumber \\
        &=(iA(X) \cdot \varphi)\vert_{L_p} + \frac{1}{2}(\tilde{E} \cdot W^{\tilde{E}}(X) \cdot \varphi)\vert_{L_p} \nonumber\\
        &=(iA(X) \cdot \varphi)\vert_{L_p} + (\tilde{E} \cdot (A(X) - g(A(X), \tilde{E})\tilde{E}) \cdot \varphi)\vert_{L_p} \nonumber\\
        &= (iA(X) \cdot \varphi)\vert_{L_p} - ((A(X) - g(A(X), \tilde{E})\tilde{E}) \cdot \tilde{E} \cdot \varphi)\vert_{L_p} \nonumber\\
        &=(iA(X) \cdot \varphi)\vert_{L_p} - (A(X) \cdot \overset{= i\varphi}{\overbrace{\tilde{E} \cdot \varphi}} + g(A(X), \tilde{E})\varphi )\vert_{L_p} \nonumber\\
        &= -g(A(X), \tilde{E})\varphi\vert_{L_p}.
\end{align}
Now, computing as in Lemma \ref{closed-conformal-lemma}, we see that $X(\esc{\varphi, \varphi}) = -2|V_{\varphi}|g(A(X), \tilde{E}) = -2\esc{\varphi, \varphi}g(A(X), \tilde{E})$, since $\varphi$ is of type I. We have obtained the formula:
\begin{align}\label{fiber-deriv}
    \frac{X(\esc{\varphi, \varphi})}{2\esc{\varphi, \varphi}}\varphi\vert_{L_p} = -g(A(X), \tilde{E})\varphi\vert_{L_p} =  \nabla^{S_{L_p}}_X\varphi\vert_{L_p}.
\end{align}
\begin{lemma}\label{parallel-fiber}
    Let $\varphi$ be a type I imaginary generalized Killing spinor on a Riemannian $\text{spin}^c$ manifold. Then the spinor $\esc{\varphi, \varphi}^{-\frac{1}{2}}\varphi\vert_{L_p}$ is $\nabla^{S_{L_p}}$-parallel, for any leaf $L_p \in \mathcal{F}$. 
\end{lemma}
\begin{proof}
    We compute that
    \begin{align*}
        \nabla^{S_{L_p}}_X\esc{\varphi, \varphi}^{-\frac{1}{2}}\varphi\vert_{L_p} &= X(\esc{\varphi, \varphi}^{-\frac{1}{2}})\varphi\vert_{L_p} + \esc{\varphi, \varphi}^{-\frac{1}{2}}\nabla^{S_{L_p}}_X\varphi\vert_{L_p} \\
        &\overset{(\ref{fiber-deriv})}{=}-\frac{X(\esc{\varphi, \varphi})}{2\esc{\varphi, \varphi}^{3/2}}\varphi\vert_{L_p} + \esc{\varphi, \varphi}^{-\frac{1}{2}}\frac{X(\esc{\varphi, \varphi})}{2\esc{\varphi, \varphi}}\varphi\vert_{L_p} = 0,
    \end{align*}
    and complete the proof.
\end{proof}
\begin{remark}\label{restriction remark}
    By (\ref{restriction-computation}), we see that the condition for a type I imaginary generalized Killing spinor to restrict to a $\nabla^{S_{L_p}}$-parallel spinor on $L_p$ is that $g(X, A(\tilde{E})) = 0$, for all $X \in T(L_p)$. In case $\varphi$ in Lemma \ref{parallel-fiber} is special, then this condition is satisfied and so $\varphi$ restricts to a $\nabla^{S_{L_p}}$-parallel spinor on $L_p$. This is the parallel spinor found in the proof of Theorem 4.1 of \cite{GrosseNakad}. 
\end{remark}
\vspace{2mm}
We can now state and prove our main theorems.
\vspace{2mm}
\begin{theorem}\label{uniform-pair-riemannian}
    Let $(M,g)$ be a Riemannian $\text{spin}^c$ manifold carrying a type I imaginary generalized Killing spinor $\varphi$. 
    \vspace{1mm}
\begin{enumerate}[label=\alph*)]
        \item For any uniform pair $(L_p', I)$ for $\Phi$, the     map $\Phi^{L_p'} : (I \times L_p', \frac{dt^2}{g(V_{\varphi},V_{\varphi})} + h_t) \to (M,g)$ 
        is a local isometry, where $h_t:= \Phi_t^*(g\vert_{L_{\Phi_t(p)}})$ is a family of metrics on $L_p'$, such that $(L_p', h_t)$ carries a parallel spinor for each $t$.
\end{enumerate}
\vspace{1mm}
Suppose in addition that $\varphi$ is special with Killing function $i\mu$.
\vspace{1mm}
\begin{enumerate}[label=\alph*)]
    \setcounter{enumi}{1} 
    \item For any uniform pair $(L_p', I)$ for $\tilde{\Phi}$, the map $\tilde{\Phi}^{L_p'} : (I \times L_p', dt^2 + e^{-4\int_{0}^{t}\mu ds} g\vert_{L_p}) \to (M,g)$ is a local isometry and $\varphi$ restricts to a parallel spinor on $(L_p', g\vert_{L_p'})$. If $\mu$ is not identically zero on the image of $\tilde{\Phi}^{L_p'}$ and is either non-negative or non-positive on the image of $\tilde{\Phi}^{L_p'}$, then the local isometry $\tilde{\Phi}^{L_p'}$ is an isometry onto its image.
\end{enumerate}  
\end{theorem}
\begin{proof}
    Everything except the last sentence follows directly from Proposition \ref{local isometry at spinors}, Lemma \ref{parallel-fiber} and Remark \ref{restriction remark}. To prove the last sentence, take a leaf $\Lambda \in \mathcal{F}$ that have non-trivial intersection with the image of $\tilde{\Phi}^{L_p'}$ and such that $\mu\vert_{\Lambda} \neq 0$. Since the flow of $\tilde{\Phi}$ preserves the foliation, there exists a $t_0 \in I$ such that $\tilde{\Phi}(t_0, L_p') \subset \Lambda$. The pair $(\tilde{\Phi}(t_0, L_p'), I - t_0)$, where $(a,b) - t_0 := (a- t_0, b-t_0)$, is uniform for $\tilde{\Phi}$ and satisfies the first injectivity condition of Proposition \ref{local isometry at spinors}. Therefore, the map $\Phi\vert_{(I - t_0) \times \tilde{\Phi}(t_0, L_p')}$ is injective and since $\tilde{\Phi}$ preserves the foliation $\mathcal{F}$, the map $\tilde{\Phi}^{L_p'}$ is also injective. Hence $\tilde{\Phi}^{L_p'}$ is an isometry onto its image.
\end{proof}
Assuming completeness of $E$ or $\tilde{E}$, we get the following global analogue of the above theorem.
\vspace{2mm}
\begin{theorem}\label{Riemannian-covering}
    Let $(M,g)$ be a connected Riemannian $\text{spin}^c$ manifold carrying a type I imaginary generalized Killing spinor. 
\vspace{1mm}
\begin{enumerate}[label=\alph*)]
        \item If $E$ is complete, then for any leaf $L_p$, we have that $\Phi^{L_p} : (\mathbb{R} \times L_p, \frac{dt^2}{g(V_{\varphi},V_{\varphi})} + h_t) \to (M,g)$ is a normal, Riemannian covering, where $h_t = \Phi_t^*(g\vert_{L_{\Phi_t(p)}})$ is a family of metrics on $L_p$ such that for each $t$, the manifold $(L_p, h_t)$ carries a parallel spinor. \label{global-a}
\end{enumerate}
\vspace{1mm}
Suppose in addition that $\varphi$ is special with Killing function $i\mu$.
\vspace{1mm}
\begin{enumerate}[label=\alph*)]
    \setcounter{enumi}{1} 
    \item If $\tilde{E}$ is complete, then for any leaf $L_p$, we have that $\tilde{\Phi}^{L_p} : (\mathbb{R} \times L_p, dt^2 + e^{-4\int_{0}^{t}\mu ds} g\vert_{L_p}) \to (M,g)$ is a normal, Riemannian covering and $\varphi$ restricts to a parallel spinor on $(L_p, g\vert_{L_p})$. If $\mu$ is not identically zero and is either non-negative or non-positive, then the covering $\tilde{\Phi}^{L_p}$ is an isometry. \label{global-b}
\end{enumerate}      
\end{theorem}
\begin{proof}
    Everything except the last sentence follows from Proposition \ref{global covering at spinors}, Lemma \ref{parallel-fiber}, and Remark \ref{restriction remark}. To prove the last statement, one appeals to the injectivity condition of Proposition \ref{global covering at spinors} and argues as in the proof of Theorem \ref{uniform-pair-riemannian}.
\end{proof}
Theorem \ref{Riemannian-covering} generalizes Theorem 4.1 and Corollary 4.2 of \cite{GrosseNakad}. In particular, it extends Baum's classification of type I imaginary Killing spinors on spin manifolds \cite[Theorem 3]{Baum89} to the $\text{spin}^c$ setting. Theorem \ref{Riemannian-covering} \ref{global-a} also generalizes a part of Theorem 4 in \cite{Leistner-Lischewski}.

\subsubsection{Type II imaginary generalized Killing spinors}
Recall that an imaginary generalized Killing spinor on a Riemannian $\text{spin}^c$ manifold is said to be of type II if the constant $Q_- = \esc{\varphi, \varphi}^2 - g(V_{\varphi}, V_{\varphi})$ is positive. In case the type II imaginary generalized Killing spinor is a Killing spinor, then Gro{\ss}e and Nakad \cite[Proposition 4.3]{GrosseNakad} showed that the manifold is isometric to hyperbolic space, provided it is complete. It was shown by Rademacher \cite{Rademacher91} that on spin manifolds, all type II special imaginary generalized Killing spinors have constant Killing function, i.e. they are Killing spinors. Grosse and Nakad \cite[Section 4.2.2]{GrosseNakad} extends this to the $\text{spin}^c$ case for dimension at least $3$ and provide a counterexample in dimension $2$.
\\
\\
To clarify matters in the case when the generalized Killing spinor need not be special, we want to consider hypersurfaces of Lorentzian manifolds. The following identifications are described in \cite{Ammann-glöckle} and \cite{Bärnotes} in the spin case and in \cite{Nakad2011} for the $\text{spin}^c$ setting. We give a quick summary similar to the above discussion on Riemannian hypersurfaces of Riemannian $\text{spin}^c$ manifolds. Let $(\overline{M}, \overline{g})$ be a time oriented Lorentzian $\text{spin}^c$ manifold with spinor bundle $S$ and let $(M, g)$ be a spacelike hypersurface with unit normal $N$. We give $(M, g)$ the induced $\text{spin}^c$ structure and write $S_M$ for the corresponding spinor bundle. When $\dim{\overline{M}}$ is odd, we have that $S_M \cong S\vert_M$ and when $\dim{\overline{M}}$ is even, we have that $S_M \cong S^+\vert_M$. In this setup, the $S^1$-bundle associated with the $\text{spin}^c$ structure on $M$ is simply the restriction $P_{S^1}\vert_{M}$, where $P_{S^1}$ is the $S^1$-bundle associated to the $\text{spin}^c$ structure on $\overline{M}$. Hence, given a connection $A \in \Omega^1(P_{S^1}, i\mathbb{R})$, we get induced connections $\nabla^S$ on $S$ and $\nabla^{S_{M}}$ on $S_{M}$. Clifford multiplication $\bullet$ on $S_M$ is related to Clifford multiplication $\cdot$ on $S$ via $X \bullet \psi = iN\cdot X \cdot \psi$. Correspondingly, the connection $\nabla^S$ relates to the connection $\nabla^{S_{M}}$, via 
\begin{align}\label{hypersurface of Lorentz}
    (\nabla^S_X\psi)_M = \nabla^{S_M}_X(\psi\vert_M) - \frac{i}{2}W^{N}(X)\bullet (\psi\vert_M).
\end{align}
The positive definite hermitian bundle metric $( \cdot,\cdot )$ in $S_M$ from \hyperref[structures]{structures on the spinor bundle} is given explicitly by $(\varphi, \psi) := \esc{N \cdot \varphi, \psi}$, where $\esc{\cdot, \cdot}$ is the scalar product in $S$. With this setup, we show that when generalized Killing spinors are defined in this generality, type II imaginary generalized Killing spinors exist in abundance.
\vspace{2mm}
\begin{theorem}\label{restriction}
    Let $(\overline{M}, \overline{g})$ be a time oriented Lorentzian $\text{spin}^c$ manifold carrying a parallel spinor $\varphi$. Then, for any spacelike hypersurface in $(\overline{M}, \overline{g})$ with unit normal $N$, $\varphi$ restricts to an imaginary generalized Killing spinor with Killing endomorphism given by $\frac{1}{2}$ times the Weingarten map of that hypersurface. If the Dirac current $V_{\varphi}$, with respect to $\varphi$ on $(\overline{M}, \overline{g})$ is timelike, then the restricted spinor is of type II, and if it is lightlike, the restricted spinor is of type I.
\end{theorem}
\begin{proof}
    By (\ref{hypersurface of Lorentz}), we have that $\varphi$ restricted to any hypersurface is an imaginary generalized Killing spinor $\phi$ with Killing endomorphism $\frac{1}{2}W^{N}$. Let $U_{\phi}$ denote the Dirac current on the hypersurface with respect to this restricted spinor. For any vector field $X$ tangent to the hypersurface, we compute that
    \begin{align*}
        \overline{g}(U_{\phi}, X) &= i(\phi, X \bullet \phi) = i\esc{N \cdot \varphi, i N \cdot X \cdot \varphi} \\
        &=\esc{N \cdot N \cdot \varphi, X \cdot \varphi} = \esc{\varphi, X \cdot \varphi} \\
        &= -\overline{g}(V_{\varphi}, X).
    \end{align*}
    Further, we see that
    \begin{align*}
        \overline{g}(V_{\varphi}, N) = -\esc{\varphi, N \cdot \varphi} = -\esc{N \cdot \varphi, \varphi} = -(\phi, \phi),
    \end{align*}
    and thus, along the hypersurface, we have that $V_{\varphi} = -U_{\phi} + (\phi, \phi)N$. If $V_{\varphi}$ is timelike, we get that 
    \begin{align*}
        0 > \overline{g}(V_{\varphi}, V_{\varphi}) = \overline{g}(U_{\phi}, U_{\phi}) + (\phi, \phi)^2\overline{g}(N,N) = \overline{g}(U_{\phi}, U_{\phi}) - (\phi, \phi)^2,
    \end{align*}
    which precisely means that $\varphi$ restricts to a type II imaginary generalized Killing spinor on the hypersurface. Similarly, if $V_{\varphi}$ is lightlike, the restricted spinor will be of type I.
\end{proof}
Lorentzian $\text{spin}^c$ manifolds carrying a parallel spinor such that the Dirac current is timelike, exists, and were studied in \cite{IKEMAKHEN2007}. The second part of the above statement is essentially known and is a motivation for studying type I imaginary generalized Killing spinors; see for example, \cite{Leistner-Lischewski}.

\backmatter






\newpage
\bibliographystyle{unsrtnat}


\end{document}